\documentclass{gtpart}     
\usepackage{pinlabel}  
\usepackage{graphicx}  
\usepackage{amscd}

\usepackage{amsmath}
\usepackage{amssymb}
\usepackage{amsthm}
\usepackage{tabularx}
\usepackage{longtable}
\usepackage{booktabs}
\usepackage{tikz-cd}
\usepackage{hyperref}

\title{Diffeomorphisms of the 4-sphere, Cerf theory and Montesinos twins}

\author{David T Gay}

\volumenumber{}
\issuenumber{}
\publicationyear{}
\papernumber{}
\startpage{}
\endpage{}
\doi{}
\MR{}
\Zbl{}
\received{}
\revised{}
\accepted{}
\published{}
\publishedonline{}
\proposed{}
\seconded{}
\corresponding{}
\editor{}
\version{}

\newtheorem{theorem}{Theorem}
\newtheorem{lemma}[theorem]{Lemma}
\newtheorem{proposition}[theorem]{Proposition}
\newtheorem{corollary}[theorem]{Corollary}

\theoremstyle{definition}
\newtheorem{definition}[theorem]{Definition}

\newtheorem{remark}[theorem]{Remark}

\newtheorem{question}[theorem]{Question}
\def\Z{\mathbb Z}
\def\N{\mathbb N}

\def\R{\mathbb R}

\newcommand{\into}{\ensuremath{\hookrightarrow}}

\newcommand{\id}{\mathop{\rm id}\nolimits}

\newcommand{\Diff}{\mathop{\rm Diff}\nolimits}
\newcommand{\Emb}{\mathop{\rm Emb}\nolimits}

\begin{document}

\begin{abstract}    

One way to better understand the smooth mapping class group of the $4$--sphere would be to give a list of generators in the form of explicit diffeomorphisms supported in neighborhoods of submanifolds, in analogy with Dehn twists on surfaces. As a step in this direction, we describe a surjective homomorphism from a group associated to loops of $2$--spheres in $S^2 \times S^2$'s onto this smooth mapping class group, discuss two natural and in some sense complementary subgroups of the domain of this homomomorphism, show that one is in the kernel, and give generators as above for the image of the other. These generators are described as twists along Montesinos twins, i.e. pairs of embedded $2$--spheres in $S^4$ intersecting transversely at two points.
\end{abstract}

\maketitle

\section{Introduction}

Given a smooth oriented manifold $X$, let $\Diff^+(X)$ be the space of orientation preserving diffeomorphisms of $X$ (fixed on a collar neighborhood of $\partial X$ if $\partial X \neq \emptyset$). Here, inspired heavily by Watanabe's work~\cite{Watanabe} on homotopy groups of $\Diff^+(B^4)$ and Budney and Gabai's work~\cite{BudneyGabai} on knotted $3$--balls in $S^4$, we initiate a study of $\pi_0(\Diff^+(S^4))$, i.e. the smooth mapping class group of the $4$--sphere. We know very little about this group except that it is abelian and that every orientation preserving diffeomorphism of $S^4$ is {\em pseudoisotopic} to the identity; the group could very well be trivial, like the topological mapping class group. Ideally we would like to find a generating set for this mapping class group defined explicitly and geometrically, for example as  explicit diffeomorphisms supported in neighborhoods of explicit submanifolds of $S^4$, in analogy with Dehn twists as generators of the mapping class groups of surfaces. In this paper we construct a surjective homomorphism from a limit of fundamental groups of certain embedding spaces of $2$--spheres in $4$--manifolds onto $\pi_0(\Diff^+(S^4))$, we describe one geometrically natural subgroup of the domain of this homomorphism which we show to be in its kernel, and we describe a ``complementary'' geometrically natural subgroup and give an explicit list of generators as above for its image in $\pi_0(\Diff^+(S^4))$.

All manifolds and maps between manifolds are smooth in this paper. The symbol $\delta_{ij}$ refers throughout to the Kronecker delta symbol.

Given manifolds $X$ and $Y$, let $\Emb(X,Y)$ denote the space of embeddings of $X$ into $Y$. For any manifold $Y$ let $Y^\dagger$ refer to the punctured manifold $Y \setminus \{p\}$ for some $p \in Y$. (The puncture is not important until the next paragraph.) Given a manifold $X$ of dimension $m$ and natural numbers $n$ and $k$, let
\[ \mathcal{S}_{k,n}(X) = \Emb(\amalg^n S^k, X \#^n (S^k \times S^{m-k})^\dagger) \] 
Note that this notation is a little ambiguous as to how many punctures are involved; we mean that the puncture happens after the connected sums, so that there is only one puncture, not $n$ punctures.

This $\mathcal{S}$ stands for ``spheres'', as in ``space of embeddings of collections of spheres'', the index $n$ is the number of spheres, and the index $k$ is the dimension of the spheres. For the main results of this paper we are interested in $X=S^4$ and $k=2$, giving us
\[ \mathcal{S}_{2,n}(S^4) = \Emb(\amalg^n S^2, S^4 \#^n (S^2 \times S^2)^\dagger) = Emb(\amalg^n S^2, \#^n (S^2 \times S^2)^\dagger) \] 
Fix a point $p \in S^2$ and let $\amalg^n (S^2 \times \{p\}) \subset \#^n (S^2 \times S^2)^\dagger$ denote the union of one copy of $S^2 \times \{p\}$ in each $S^2 \times S^2$ summand of $\#^n (S^2 \times S^2)^\dagger$. This will be our basepoint in $\mathcal{S}_{2,n}(S^4)$, which we will often suppress from our notation, with the understanding that $\mathcal{S}_{2,n}(S^4)$ is a pointed space. We will also be interested in two subspaces of $\mathcal{S}_{2,n}(S^4)$: Let $\mathcal{S}^0_{2,n}(S^4)$ denote the subspace of embeddings with the property that for each $i$ and $j$ the $i$'th component of $\amalg^n S^2$ intersects the $\{p\} \times S^2$ in the $j$'th summand of $\#^n (S^2 \times S^2)^\dagger$ transversely at $\delta_{ij}$ points. Let $\widehat{\mathcal{S}}_{2,n}(S^4)$ denote the subspace of embeddings with the property that the image of $\amalg^n S^2$ is disjoint from $\amalg^n (S^2 \times \{p'\})$ for some fixed $p' \neq p \in S^2$. Note that our basepoint lies in both of these subspaces.

Thanks to the puncture, we have a natural basepoint-preserving inclusion $\jmath: \mathcal{S}_{2,n}(S^4) \into \mathcal{S}_{2,n+1}(S^4)$, respecting the inclusions of $\mathcal{S}^0_{2,n}(S^4)$ and $\widehat{\mathcal{S}}_{2,n}(S^4)$, and thus inclusion-induced homomorphisms $\jmath_* : \pi_1(\mathcal{S}_{2,n}(S^4)) \to \pi_1(\mathcal{S}_{2,n+1}(S^4))$ which commute with 
the inclusion-induced homomorphisms $\imath_*: \pi_1(\mathcal{S}^0_{2,n}(S^4)) \to \pi_1(\mathcal{S}_{2,n}(S^4))$ and $\imath_*: \pi_1(\widehat{\mathcal{S}}_{2,n}(S^4)) \to \pi_1(\mathcal{S}_{2,n}(S^4))$. As a consequence, we have limit groups which we will denote $\pi_1(\mathcal{S}_{2,\infty}(S^4))$, $\pi_1(\mathcal{S}^0_{2,\infty}(S^4))$, $\pi_1(\widehat{\mathcal{S}}_{2,\infty}(S^4))$ (it is not important for us to think about the limiting spaces, just the groups, but this notation is convenient) and limiting inclusion-induced homomorphisms $\imath_*$ between them.
Our first result is:
\begin{theorem} \label{T:H2infSurj}
 There exists a sequence of homomorphisms $\mathcal{H}_{2,n}: \pi_1(\mathcal{S}_{2,n}(S^4)) \to \pi_0(\Diff^+(S^4))$, for $n \in \N$, satisfying the following properties:
 \begin{enumerate}
  \item The $\mathcal{H}_{2,n}$'s commute with the $\jmath_*$'s, fitting into the commutative diagram shown in Figure~\ref{F:SnHnCommutative}. Thus this gives rise to a limit homomorphism $\mathcal{H}_{2,\infty} : \pi_1(\mathcal{S}_{2,\infty}(S^4)) \to \pi_0(\Diff^+(S^4))$ and the following diagram:
  {\setlength\mathsurround{0pt}
  \[
  \begin{tikzcd}[row sep=small, column sep=small, cramped]
  \pi_1(\mathcal{S}^0_{2,\infty}(S^4)) \arrow[dr, "\imath_*" near start] &&&& \\
  & \pi_1(\mathcal{S}_{2,\infty}(S^4)) \arrow[rrr, "\mathcal{H}_{2,\infty}"] &&& \pi_0(\Diff^+(S^4))\\
  \pi_1(\widehat{\mathcal{S}}_{2,\infty}(S^4)) \arrow [ur, "\imath_*"'] &&&& 
  \end{tikzcd}
  \]
  }
  \item The limit homomomorphism $\mathcal{H}_{2,\infty}$ is surjective.
  \item The image of $\pi_1(\mathcal{S}^0_{2,\infty}(S^4))$ under $\imath_*$ is contained in the kernel of $\mathcal{H}_{2,\infty}$.
 \end{enumerate}
\end{theorem}

\begin{figure}
{\setlength\mathsurround{0pt}
\[
\begin{tikzcd}[row sep=small, column sep=small, cramped]
\pi_1(\mathcal{S}^0_{2,n-1}(S^4)) \arrow[ddd, "\jmath_*"] \arrow[drr, "\imath_*" near start] &&& \\
&& \pi_1(\mathcal{S}_{2,n-1}(S^4)) \arrow[ddd, "\jmath_*"] \arrow[dddr, "\mathcal{H}_{2,n-1}"] & \\
& \pi_1(\widehat{\mathcal{S}}_{2,n-1}(S^4)) \arrow [ur, "\imath_*"] && \\
\pi_1(\mathcal{S}^0_{2,n}(S^4)) \arrow[ddd, "\jmath_*"] \arrow[drr, "\imath_*" near start] &&& \\
&& \pi_1(\mathcal{S}_{2,n}(S^4)) \arrow[ddd, "\jmath_*"] \arrow[r, "\mathcal{H}_{2,n}"] & \pi_0(\Diff^+(S^4))\\
& \pi_1(\widehat{\mathcal{S}}_{2,n}(S^4)) \arrow[from=uuu, crossing over, "\jmath_*" near start] \arrow [ur, "\imath_*"] && \\
\pi_1(\mathcal{S}^0_{2,n+1}(S^4)) \arrow[drr, "\imath_*" near start] &&& \\
&& \pi_1(\mathcal{S}_{2,n+1}(S^4)) \arrow[uuur, "\mathcal{H}_{2,n+1}"'] & \\
& \pi_1(\widehat{\mathcal{S}}_{2,n+1}(S^4)) \arrow[from=uuu, crossing over, "\jmath_*" near start] \arrow [ur, "\imath_*"] &&
\end{tikzcd}
\]
}
\caption{Commutative diagram relating maps in Theorem~\ref{T:H2infSurj}. \label{F:SnHnCommutative}}
\end{figure}

Note that the diagrams of our theorem include $\pi_1(\widehat{\mathcal{S}}_{2,n}(S^4))$ and $\pi_1(\widehat{\mathcal{S}}_{2,\infty}(S^4))$ even though the actual results in the theorem do not reference these groups. This is to set the context for Theorem~\ref{T:TwinsGenerate}, which gives explicit generators for the image of $\pi_1(\widehat{\mathcal{S}}_{2,\infty}(S^4))$ in $\pi_0(\Diff^+(S^4))$.

Krannich and Kupers~\cite{KrannichKupers} have given an alternative, shorter proof of a generalization of parts of Theorem~\ref{T:H2infSurj}, which explicitly connects this result to work of Kreck~\cite{Kreck} and Quinn~\cite{Quinn}.

As pointed out by the referee, Theorem~10.1 of~\cite{BudneyGabai} includes the statement that $\pi_1(\Emb(B^2,S^2 \times B^2))$ is free abelian of rank two, with explicit generators given, and one might wonder what happens to these generators under the homomorphism $\mathcal{H}_{2,1}$ when the $B^2$ is capped off to an $S^2$. However, the boundary condition for $\Emb(B^2,S^2 \times B^2)$ here is $\{p\} \times S^1$ for some fixed $p \in S^2$. If such an embedding of $B^2$ into $S^2 \times B^2$ is capped off to an embedding of $S^2$ into $S^2 \times S^2$ with any embedding of $B^2$ into the $S^2 \times B^2$ with the same boundary condition, the resulting embedding remains geometrically dual to $S^2 \times \{p\}$ for any $p \in S^1 = \partial B^2$. Thus all of these potentially interesting loops of $S^2$'s in $S^2 \times S^2$ land in $\pi_1(\mathcal{S}^0_{2,\infty}(S^4))$ and hence in the kernel of $\mathcal{H}_{2,\infty}$.

The homomorphisms $\mathcal{H}_{2,n}$ will be defined precisely in Section~\ref{S:LoopsToDiffs}, but they can be briefly described as follows: Note that surgery along our basepoint embedding of $\amalg^n S^2$ in $\#^n (S^2 \times S^2)$ turns $\#^n (S^2 \times S^2)$ into $S^4$. Thus a loop of embeddings starting and ending at this basepoint can be interpreted as an $S^1$--parameterized family of surgeries all of which turn $\#^n (S^2 \times S^2)$ into a $4$--manifold diffeomorphic to $S^4$. These fit together to give an $S^4$--bundle over $S^1$. The monodromy of this bundle is the output of $\mathcal{H}_{2,n}$. If one is worried about the fact that the monodromy of a bundle is only well defined up to conjugation, either note that $\pi_0(\Diff^+(S^4))$ is abelian or note that we can establish a canonical (up to isotopy) diffeomorphism between $S^4$ and the result of surgering $\#^n (S^2 \times S^2)$ along the basepoint embedding.

The proof of Theorem~\ref{T:H2infSurj} uses the fact that every orientation preserving diffeomorphism of $S^4$ is pseudoisotopic to the identity, Cerf's technique~\cite{Cerf} to turn a pseudoisotopy into a family of Morse functions, and results of Hatcher and Wagoner~\cite{HatcherWagoner} to optimally clean up such a family and its associated family of handle decompositions.

Our second result characterizes the image of $\pi_1(\widehat{\mathcal{S}}_{2,\infty}(S^4))$ in $\pi_0(\Diff^+(S^4))$ in terms of a countable list of explicit generators which we will now describe. This could in principle be all of $\pi_0(\Diff^+(S^4))$, although we have no evidence for or against that possibility.

A {\em Montesinos twin} in $S^4$ is a pair $W=(R,S)$ of embedded $2$--spheres $R,S \subset S^4$ intersecting transversely at two points. For us, the $2$--spheres are both oriented. Montesinos~\cite{MontesinosI} shows, and we explain in Section~\ref{S:Montesinos}, that the boundary of a regular neighborhood $\nu(W)$ of $R \cup S$ is diffeomorphic to $S^1 \times S^1 \times S^1$ and that in fact there is a canonical parametrization $S^1_l \times S^1_R \times S^1_S \cong \partial \nu(W)$. This parametrization is characterized by $S^1_l \times \{b\} \times \{c\}$ being homologically trivial in $H_1(S^4 \setminus (R \cup S))$, i.e. a ``longitude'', $\{a\} \times S^1_R \times \{c\}$ being a meridian for $R$, and $\{a\} \times \{b\} \times S^1_S$ being a meridian for $S$. This parametrization is canonical up to postcomposing with diffeomorphisms of $\partial \nu(W)$ which are isotopic to the identity and precomposing with independent diffeomorphisms of $S^1_l$, $S^1_R$ and $S^1_S$. This then parametrizes a regular neighborhood of $\partial \nu(W)$ as $[-1,1] \times S^1_l \times S^1_R \times S^1_S$. We adopt the orientation conventions that $S^1_R$ and $S^1_S$ have the standard meridian orientations coming from the orientations of $R$ and $S$, that $[-1,1]$ is oriented in the outward direction from $\nu(W)$, and that $S^1_l$ is oriented so that the orientation of $[-1,1] \times S^1_l \times S^1_R \times S^1_S$ agrees with the standard orientation of $S^4$.

\begin{definition}
 Given a Montesinos twin $W$ in $S^4$, parametrize a neighborhood of $\partial \nu(W)$ as $[-1,1] \times S^1_l \times S^1_R \times S^1_S$ as above. Let $\tau_l: [-1,1] \times S^1_l \to [-1,1] \times S^1_l$ denote a right-handed Dehn twist. The {\em twin twist along $W$}, denoted $\tau_W$, is the diffeomorphism of $S^4$ which is the identity outside this neighborhood of $\partial \nu(W)$ and is equal to $\tau_l \times \id_{S^1_R} \times \id_{S^1_S}$ inside this neighborhood.
\end{definition}
By the canonicity of our parametrization, $\tau_W$ is well-defined up to isotopy, i.e. $[\tau_W]$ is a well-defined class in $\pi_0(\Diff^+(S^4))$. Incidentally, we have the following as a consequence of our orientation conventions:
\begin{lemma} \label{L:TwinFlip}
 If $W = (R,S)$ is a Montesinos twin, then $[\tau_W]^{-1} = [\tau_{(S,R)}] = [\tau_{(\overline{R},S)}] = [\tau_{(R,\overline{S})}]$.
\end{lemma}

\begin{proof}
 Either switching the spheres $S$ and $R$, or reversing the orientation of one of them, reverses the orientation of $S^1_R \times S^1_S$, which then forces the reversal of the orientation of $[-1,1] \times S^1_l$, changing a positive Dehn twist to a negative Dehn twist.
\end{proof}

We now describe a family of twins $W(i) = (R(i),S)$, for $i \in \mathbb{N} \cup \{0\}$. Figure~\ref{F:snake} illustrates $W(3)$; $W(i)$ is the same but with $i$ turns in the spiral rather than $3$ turns. Figure~\ref{F:Ribbon} and~\ref{F:FingerMove} give two alternate descriptions of this twin. Orientations are not made explicit here since our main claim is simply that the twists invoved generate a certain group, and the inverse of a generator is still a generator. Note that both $R(i)$ and $S$ are individually unknotted $2$--spheres, and that the twin $W(0)$ is the trivial ``unknotted twin''.
\begin{figure}
  \labellist
  \small\hair 2pt
  \pinlabel $R(3)$ [r] at 2 125
  \pinlabel $S$ [l] at 260 106
  \endlabellist
  \centering
  \includegraphics[width=8cm]{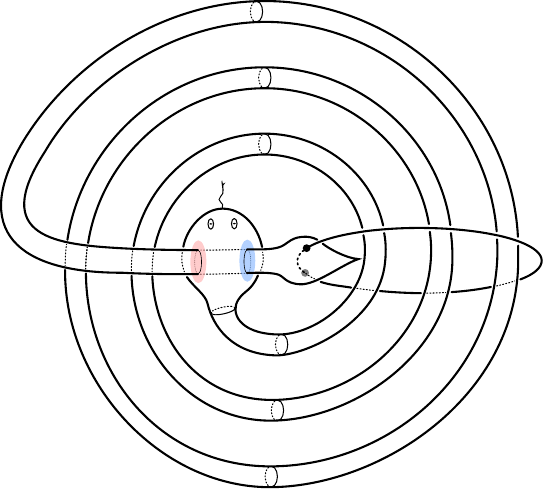}
  \caption{An illustration of $W(3) = (R(3),S)$. The picture mostly happens in the slice $\{t=0\} \subset \R^4 = \{(x,y,z,t)\} \subset \R^4 \cup \{\infty\}=S^4$. The ring labelled $S$ is a slice through $S$, which shrinks to a point as we move forwards and backwards in the ``time'' coordinate $t$. The ``snake whose tail passes through his head'' is $R(3)$, which is projected onto $\{t=0\}$, intersecting itself along one circle in the middle of the red disk (the ``snake's left ear hole'') and along another circle in the middle of the blue disk (the ``right ear hole''). Blue and red indicate that these disks are pushed slightly forwards (blue) and backwards (red) in time to resolve these intersections; otherwise $R(3)$ lies in the slice $\{t=0\}$. The two points of intersection between $R(3)$ and $S$ are where the ring pierces the tail.}
  \label{F:snake}
\end{figure}
\begin{figure}
 \centering
 \includegraphics[width=8cm]{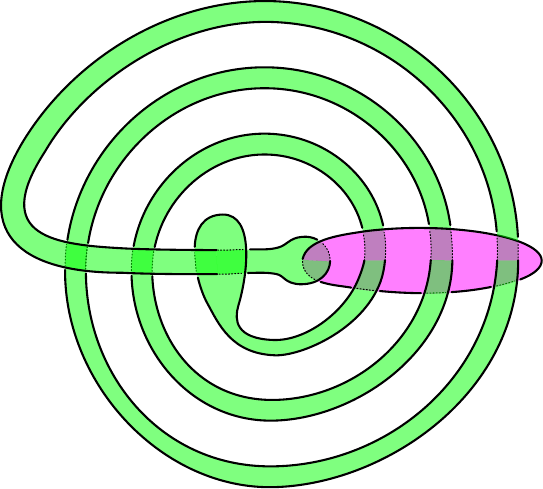}
 \caption{Another illustration of $W(3)=(R(3),S)$. Here we have drawn an immersed pair of disks, one green and one pink, with mostly ribbon intersections except for one nonribbon arc. Pushing these two disks into $\R^4_+$ or $\R^4_-$ and resolving ribbon intersections in the usual way gives two embedded disks intersecting each other transversely once, and then taking one copy in $\R^4_+$ and one in $\R^4_-$ glued along their common boundary, i.e. doubling the ribbon disks, gives $R(3)$ (green) and $S$ (pink) in $\R^4 \subset S^4$.}
 \label{F:Ribbon}
\end{figure}
\begin{figure}
 \centering
 \includegraphics{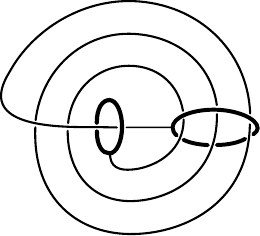}
 \caption{Yet another illustration of $W(3)$. Here we have drawn two disjoint, embedded $2$--spheres in $S^4$ (the two thick circles, becoming $2$--spheres when shrunk to points forwards and backwards in time) and an arc connecting them. Push a finger from one of the spheres out along this arc and then do a finger move when you encounter the other sphere, creating a pair of transverse intersections, and the result is $W(3)$.}
 \label{F:FingerMove}
\end{figure}

Our second result is:
\begin{theorem} \label{T:TwinsGenerate}
 The subgroup $\mathcal{H}_{2,\infty}(\imath_*(\pi_1(\widehat{\mathcal{S}}_{2,\infty}(S^4))))$ of $\pi_0(\Diff^+(S^4))$ is generated by the twin twists $\{\tau_{W(i)} \mid i \in \mathbb{N}\}$. In other words, every diffeomorphism of $S^4$ coming from a loop of embeddings of $\amalg^n S^2$ into $\#^n (S^2 \times S^2)$ which starts and ends at $\amalg^n (S^2 \times \{p\})$ and remains disjoint from a parallel embedding $\amalg^n (S^2 \times \{p'\})$ is isotopic to a composition of twin twists along the Montesinos twins $\{W(i)\}$.
\end{theorem}
In fact these automorphisms $\tau_{W(i)}$ are also examples of the ``barbell maps'' discussed in~\cite{BudneyGabai}; readers familiar with barbell maps should be able to use the description of $W(i)$ in Figure~\ref{F:FingerMove} to see the connection.

To avoid cumbersome notation, we will mostly refer to $\mathcal{H}_{2,\infty}(\imath_*(\pi_1(\widehat{\mathcal{S}}_{2,\infty}(S^4))))$ as ``the image of $\pi_1(\widehat{\mathcal{S}}_{2,\infty}(S^4))$''.
\begin{question}
 Is $[\tau_W]$ in the image of $\pi_1(\widehat{\mathcal{S}}_{2,\infty}(S^4))$ for an arbitrary Montesinos twin $W$? More generally, given any embedding of $S^1 \times \Sigma \into S^4$, for closed surface $\Sigma$, a tubular neighborhood gives an embedding of $[-1,1] \times S^1 \times \Sigma \into S^4$ which gives a diffeomorphism $\tau \times \id_\Sigma$, where $\tau$ is the Dehn twist on $[-1,1] \times S^1$. Are such diffeomorphisms always in the image of $\pi_1(\widehat{\mathcal{S}}_{2,\infty}(S^4))$?
\end{question}
One could try to answer these questions either through Cerf theory, by explicitly identifying a pseudoisotopy from a given diffeomorphism of $S^4$ to the identity, and then extracting a loop of attaching spheres for $5$--dimensional $2$--handles, or one could try to work explicitly with the diffeomorphisms in $S^4$ and try to find relationships amongst such twists, to relate them to twists along our standard Montesinos twins $W(i)$.

The bigger questions are the following, with affirmative answers to both showing that the smooth mapping class group of $S^4$ is trivial:
\begin{question}
 Is the image of $\pi_1(\widehat{\mathcal{S}}_{2,\infty}(S^4))$ trivial?
\end{question}
Theorem~\ref{T:TwinsGenerate} could help prove this if one can exhibit explicit isotopies from $\tau_{W(i)}$ to $\id_{S^4}$.
\begin{question}
 Is the image of $\pi_1(\widehat{\mathcal{S}}_{2,\infty}(S^4))$ equal to all of $\pi_0(\Diff^+(S^4))$?
\end{question}
Since $\imath_*(\pi_1(\mathcal{S}^0_{2,\infty}(S^4)))$ is in the kernel of $\mathcal{H}_{2,\infty}$, we know that $\mathcal{H}_{2,\infty}$ factors through the quotient map $\pi: \pi_1(\mathcal{S}_{2,\infty}(S^4)) \to \pi_1(\mathcal{S}_{2,\infty}(S^4))/\langle\imath_*(\pi_1(\mathcal{S}^0_{2,\infty}(S^4)))\rangle$, where $\langle H \rangle$ denotes the normal closure of a subgroup $H$. Thus one way to answer the above question in the affirmative is to show that 
\[\pi \circ \imath_* : \pi_1(\widehat{\mathcal{S}}_{2,\infty}(S^4)) \to \pi_1(\mathcal{S}_{2,\infty}(S^4))/\langle\imath_*(\pi_1(\mathcal{S}^0_{2,\infty}(S^4)))\rangle\] 
is surjective. On the other hand this does not need to be true for the answer to this question to be ``yes'', since the kernel of $\mathcal{H}_{2,\infty}$ could presumably be much larger than $\imath_*(\pi_1(\mathcal{S}^0_{2,\infty}(S^4)))$.

In the next section we elaborate on the connection between loops of embeddings of certain spheres and self-diffeomorphisms of other spheres, setting up the general theory in various dimensions and codimensions and defining the homomorphisms $\mathcal{H}_{2,n}$. After that, we devote one section to completing the proof of Theorem~\ref{T:H2infSurj}, and we break the proof of Theorem~\ref{T:TwinsGenerate} into the three remaining sections.

The author would like to thank Bruce Bartlett, Sarah Blackwell, Mike Freedman, Robert Gompf, Jason Joseph, Danica Kosanovic, Peter Lambert-Cole, Gordana Matic, Benjamin Ruppik, Rob Schneiderman, Peter Teichner and Jeremy Van Horn-Morris for helpful conversations along the way, and most especially Hannah Schwartz for pointing out Montesinos's work and Dave Gabai for initial inspiration and for pointing out the mistake in the first version that claimed far too much, and both Hannah and Dave for many clarifying conversations about loops of spheres.

Much of this work was carried out during the author's year at the Max Planck Institute for Mathematics in Bonn, and thus the author gratefully acknowledges the institute's support in the form of the 2019-2020 Hirzebruch Research Chair. This work was also supported by individual NSF grant DMS-2005554 and NSF Focused Research Group grant DMS-1664567.

Lastly, the author would like to thank the referee for a very careful reading and very helpful comments which led to greatly improved exposition.

\section{From loops of spheres to diffeomorphisms} \label{S:LoopsToDiffs}
In this section we define the homomorphisms $\mathcal{H}_{2,n}: \pi_1(\mathcal{S}_{2,n}(S^4)) \to \pi_0(\Diff^+(S^4))$ and establish some lemmas that prove all the results of Theorem~\ref{T:H2infSurj} except the surjectivity of $\mathcal{H}_{2,\infty}$. This surjectivity will be proved in the next section.

We define $\mathcal{H}_{2,n}$  by first defining a more general family of homomorphisms turning loops of {\em framed} embeddings of spheres of various dimensions into bundles of cobordisms and hence into self-diffeomorphisms of smooth manifolds. In the introduction above, we had $2$--spheres embedded in $4$--manifolds, but we did not mention framings of these $2$--spheres. Below, we will work with framed spheres and then at the end of the section when we relate this back to the terminology of the introduction, we will see why we can ignore the framing issues. In addition to adding framings, we allow the spheres in the domain to have arbitrary dimension, and we allow the target spaces of the embeddings to have arbitrary dimension.

Fix an $m$--manifold $X$, for some $m \geq 2$, fix integers $0 < k < m$ and $n \geq 0$, and let $\#^n (S^k \times S^{m-k})^\dagger$ denote a punctured $\#^n (S^k \times S^{m-k})$. As in the introduction, the puncture is needed so that we can view $\#^n (S^k \times S^{m-k})^\dagger$ as a subspace of $\#^{n+1} (S^k \times S^{m-k})^\dagger$. Now consider the following space of embeddings of collections of {\em framed} spheres:
\[ \mathcal{FS}_{k,n}(X) = \Emb(\amalg^n (S^k \times B^{m-k}),X \#^n (S^k \times S^{m-k})^\dagger) \]
In the notation, $\mathcal{FS}$ stands for ``framed spheres'', the parameter $k$ tells us that the spheres are $k$--spheres, the subscript $n$ tells us how many $k$--spheres and how many $(S^k \times S^{m-k})^\dagger$ summands there are, and $X$ is the base $m$--manifold. Picking a fixed point $p \in S^{m-k}$ and a disk neighborhood $U$ of $p$ parametrized as $B^{m-k}$, we get a natural basepoint $\amalg^n (S^k \times U) \subset X \#^n (S^k \times S^{m-k})^\dagger$ for $\mathcal{FS}_{k,n}(X)$, which we will again generally suppress in our notation, understanding that $\mathcal{FS}_{k,n}(X)$ is a pointed space.

We now define a homomorphism
\[ \mathcal{FH}_{k,n} : \pi_1(\mathcal{FS}_{k,n}(X)) \to \pi_0(\Diff^+(X)) \]
as follows: Represent an element $b$ of $\pi_1(\mathcal{FS}_{k,n}(X))$ by a smoothly varying $1$--parameter family of embeddings 
\[ \beta_t: \amalg^n (S^k \times B^{m-k}) \into X \#^n (S^k \times S^{m-k})^\dagger \subset X \#^n (S^k \times S^{m-k}) \]
with $\beta_0 = \beta_1 = \amalg^n (S^k \times U)$. Use this to define a smooth embedding
\[ \overline{\beta}: S^1 \times \amalg^n (S^k \times B^{m-k}) \into S^1 \times X \#^n (S^k \times S^{m-k}) \]
via $\overline{\beta}(t,p) = (t,\beta_t(p))$, identifying $S^1$ with $[0,1]/1\!\sim\! 0$.
Now perform fiberwise surgery along $\overline{\beta}$, i.e remove 
\[\overline{\beta}(S^1 \times \amalg^n (S^k \times \mathring{B}^{m-k}))\]
and replace with 
\[ S^1 \times \amalg^n(B^{k+1} \times S^{m-k-1})\] 
via the gluing map
\[ \overline{\beta}: S^1 \times \amalg^n (S^k \times S^{m-k-1}) \into S^1 \times X \#^n (S^k \times S^{m-k}) \]
Let $Y$ denoting the resulting $(m+1)$--manifold; because the surgery respects the $S^1$--factor, $Y$ is a bundle over $S^1$. The fiber over $0$ is equal to the result of surgering $X \#^n (S^k \times S^{m-k})$ along $\amalg^n (S^k \times B^{m-k})$. Since the complement of $S^k \times B^{m-k}$ in $S^k \times S^{m-k}$ is canonically identified with $S^k \times B^{m-k}$, the result of surgering $S^k \times S^{m-k}$ along $S^k \times B^{m-k}$ is canonically identified with $(S^k \times B^{m-k}) \cup (B^{k+1} \times S^{m-k-1})$, which is canonically identified with $S^m$ as the boundary of $B^{k+1} \times B^{m-k}$. Thus the fiber over $0$ can be canonically identified with $X$, and thus the monodromy of $Y$ is a well-defined element of $\pi_0(\Diff^+(X))$. This element of $\pi_0(\Diff^+(X))$ is our definition of $\mathcal{FH}_{k,n}(b)$, where $b \in \pi_1(\mathcal{FS}_{k,n}(X))$ is the element represented by the family $\beta_t$.

We now give an equivalent definition of $\mathcal{FH}_{k,n}$ in terms of parameterized handle attachment rather than parameterized surgery, more in line with Cerf theory and more useful for the rest of this paper.

The reader may find it helpful to think of $S^1$--parameterized handle attachment as attachment of {\em round handles}, as introduced by~\cite{Asimov} and motivated presumably by earlier work by Bott~\cite{Bott}. In the dimension and indices we care about, a round $(m+2)$--dimensional $(k+1)$--handle is $S^1 \times B^{k+1} \times B^{m-k}$ attached along $S^1 \times S^k \times B^{m-k}$. If one attaches a $(m+2)$--dimensional round $(k+1)$--handle to an $(m+2)$--dimensional manifold $Z$ which is itself equipped with a fibration $\pi: Z \to S^1$, and if the attaching map $\gamma: S^1 \times S^k \times B^{m-k} \to \partial Z$ respects the fibration in the sense that $\pi \circ \gamma$ is the projection map $S^1 \times S^k \times B^{m-k} \to S^1$, then the resulting manifold $Z'$ again fibers over $S^1$ so that each fiber of $Z'$ is obtained from the corresponding fiber of $Z$ by attaching a standard $(m+1)$--dimensional $(k+1)$--handle to $Z$. 

As in the parameterized surgery discussion, given $b \in \pi_1(\mathcal{FS}_{k,n}(X))$, represent $b$ by the family of embeddings $\beta_t$ and use this to produce the embedding
\[ \overline{\beta}: S^1 \times \amalg^n (S^k \times B^{m-k}) \into S^1 \times X \#^n (S^k \times S^{m-k}) \]
We will build a $(m+2)$--manifold $Z$ which fibers over $S^1$ and is a fiberwise cobordism from $S^1 \times X$ to the $(m+1)$--dimensional bundle $Y$ constructed in the preceding paragraph. By a ``fiberwise cobordism'', see Figure~\ref{F:fiberwisecobordism}, we mean that $\partial Z = -(S^1 \times X) \amalg Y$ and that the bundle maps $Z \to S^1$, $S^1 \times X \to S^1$ and $Y \to S^1$ all commute with the inclusions of $S^1 \times X$ and $Y$ into $Z$, so that the fiber $Z_t$ of $Z$ over some $t \in S^1$ is itself a cobordism from $\{t\} \times X$ to the fiber $Y_t$ of $Y$ over $t$.

\begin{figure}
  \labellist
  \small\hair 2pt
  \pinlabel $S^1$ [r] at 2 23
  \pinlabel $t$ at 114 8
  \pinlabel $X$ [br] at 114 60
  \pinlabel $Y_t$ [tr] at 114 142
  \pinlabel $Z_t$ [l] at 114 110
  \pinlabel {\rotatebox{-6}{$S^1 \times X$}} [b] at 50 57
  \pinlabel $Y$ [b] at 70 143
  \pinlabel $Z$ at 30 115
  \endlabellist
  \centering
  \includegraphics{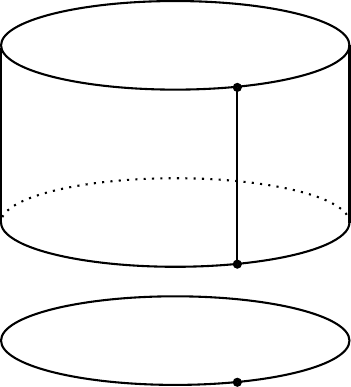}
  \caption{A fiberwise cobordism $Z$ from $S^1 \times X$ to $Y$, with fiber $Z_t$ being a cobordism from $X$ to $Y_t$.}
  \label{F:fiberwisecobordism}
\end{figure}

To build $Z$, first let $W$ equal $[0,1] \times X$ with $n$ $(m+1)$--dimensional $k$--handles attached along $n$ unlinked $0$--framed unknotted $S^{k-1}$'s lying in a ball in $\{1\} \times X \subset [0,1] \times X$. Thus $W$ is a cobordism from $X$ to $X \#^n (S^k \times S^{m-k})$. Now consider $S^1 \times W$, which is a cobordism from $S^1 \times X$ to $S^1 \times X\#^n (S^k \times S^{m-k})$. Let $Z$ be the result of attaching $n$ {\em round} $(m+2)$--dimensional $(k+1)$--handles to the top boundary $S^1 \times X\#^n (S^k \times S^{m-k})$ of $S^1 \times W$ using $\overline{\beta}$ as the attaching map. Equivalently, use each $\beta_t$ as the attaching map for $n$ $(m+1)$--dimensional $(k+1)$--handles attached to $\{t\} \times X \#^n (S^k \times S^{m-k}) \subset \{t\} \times Y$; interpreting this as smoothly varying fiberwise handle attachment gives our construction of $Z$. The top boundary of $Z$ is now a bundle over $S^1$. Since, at $t=0$, the atttaching maps $\beta_0$ for the $n$ $(k+1)$--handles are in standard cancelling position with respect to the $n$ $k$--handles used to build $W$, we see that the fiber of $Z$ over $0$ is canonically identified with $[0,1] \times X$, and thus the top boundary of this fiber is canonically identified with $X$. Thus the monodromy of the top boundary of $Z$ is a well-defined element of $\pi_0(\Diff^+(X))$, which we define to be $\mathcal{FH}_{k,n}(b)$.

The fact that these two definitions of $\mathcal{FH}_{k,n}(b)$ agree is simply because, just as handle attachment modifies the boundary of a manifold by surgery, parameterized handle attachment modifies the boundary of a bundle by parameterized surgery.

Thanks to the punctures, we have basepoint preserving inclusions
\[ \ldots \subset \mathcal{FS}_{k,n}(X) \subset \mathcal{FS}_{n+1}(X,k) \subset \ldots \]
and thus induced maps on $\pi_1$ and thus a direct limit
\[ \ldots \to \pi_1(\mathcal{FS}_{k,n}(X)) \to \pi_1(\mathcal{FS}_{k,n+1}(X)) \to \ldots \to \pi_1(\mathcal{FS}_{k,\infty}(X)) \]
Again, we do not really care about the limiting spaces, just the groups. Thus one should think of an element of $\pi_1(\mathcal{FS}_{k,\infty}(X))$ as an equivalence class of loops in some $\mathcal{FS}_{k,n}(X)$, where two such loops are equivalent if they become homotopic after including into some $\mathcal{FS}_{k,N}(X)$ for some $N  \geq n$. It is not hard to see that these induced maps on $\pi_1$ commute with the $\mathcal{FH}_{k,n}$ homomorphisms, so that we can take the direct limit of the homomorphisms $\mathcal{FH}_{k,n}$

We summarize our discussion thus far as follows:
\begin{definition} \label{D:FHn}
 Given an $m$--manifold $X$ with $m \geq 2$, given an integer $k$ with $0 < k < m$, and an integer $n \geq 0$, the associated {\em parameterized framed handle attachment homomomorphism} is the homomomorphism
 \[ \mathcal{FH}_{k,n} : \pi_1(\mathcal{FS}_{k,n}(X)) \to \pi_0(\Diff^+(X)) \]
 defined in the following way: Use $b \in \pi_1(\mathcal{FS}_{k,n}(X))$ to build a bundle over $S^1$ in which each fiber is a cobordism from $X$ to some $m$--manifold built with $n$ standardly attached $(m+1)$--dimensional $k$--handles and $n$ $(m+1)$--dimensional $(k+1)$--handles attached according to a loop of embeddings representing $b$. The monodromy of the top boundary of this fiberwise cobordism is $\mathcal{FH}_{k,n}(b)$. Taking the direct limit as $n$ goes to $\infty$ gives the {\em limit homomorphism}
 \[ \mathcal{FH}_{k,\infty} : \pi_1(\mathcal{FS}_\infty(X,k)) \to \pi_0(\Diff^+(X)) \]
\end{definition}

As in the unframed setting of the introduction, we have two natural subspaces of $\mathcal{FS}_{k,n}(X)$: Let $\mathcal{FS}^0_{k,n}(X)$ denote those embeddings of $\amalg^n (S^k \times B^{m-k}$) into $X \#^n (S^k \times S^{m-k})$ for which the $S^k \times \{0\}$ in the $i$'th $S^k \times B^{m-k}$ transversely intersects the $\{p\} \times S^{m-k}$ in the $j$'th $S^k \times S^{m-k}$ summand transversely at $\delta_{ij}$ points. Let $\widehat{\mathcal{FS}}_{k,n}(X)$ denote the subspace of embeddings with the property that the image of $\amalg^n (S^k \times B^{m-k})$ is disjoint from $\amalg^n (S^k \times \{p'\})$ for some fixed $p' \in S^{m-k} \setminus U$. Note that our basepoint lies in both of these subspaces.

\begin{proposition}
 For each $n$, $\imath_*(\pi_1(\mathcal{FS}^0_{k,n}(X)))$ is in the kernel of $\mathcal{FH}_{k,n}$. Thus in the limit, $\imath_*(\pi_1(\mathcal{FS}^0_{k,\infty}(X)))$ is in the kernel of $\mathcal{FH}_{k,\infty}$.
\end{proposition}
\begin{proof}
 This is because when the $S^k \times \{0\}$'s are dual to the $\{p\} \times S^{m-k}$'s for all $t$ in a loop of embeddings $\beta_t$, then for all $t$ the $k$--handles and $(k+1)$--handles cancel uniquely (this is Cerf's {\em l'unicit\'e de mort}~\cite{Cerf}). Thus the cobordism $Z$ from $S^1 \times X$ to $Y$ becomes a bundle over $S^1$ with a Morse function without critical points which restricts to each fiber as a Morse function without critical points. This implies that the monodromy at the top is isotopic to the monodromy at the bottom, and since the monodromy of $S^1 \times X$ is the identity then the monodromy of $Y$ is isotopic to the identity.
\end{proof}
 
Now we discuss the relationship to spaces of spheres without framings. Recall that
\[ \mathcal{S}_{k,n}(X) = \Emb(\amalg^n S^k, X \#^n (S^k \times S^{m-k})^\dagger) \]
There is an obvious ``framing forgetting'' map of pointed spaces
\[ \mathcal{F} : \mathcal{FS}_{k,n}(X) \to \mathcal{S}_{k,n}(X) \]
given by restricting an embedding of $S^k \times B^{m-k}$ to $S^k = S^k \times \{0\}$. We have obvious definitions of the subspaces $\mathcal{S}^0_{k,n}(X)$ and $\widehat{\mathcal{S}}_{k,n}(X)$.

\begin{lemma}
 The parameterized framed handle attachment homomorphism 
 \[ \mathcal{FH}_{k,n} : \pi_1(\mathcal{FS}_{k,n}(X)) \to \pi_0(\Diff^+(X)) \]
 factors through the image of the homomorphism
 \[ \mathcal{F}_* : \pi_1(\mathcal{FS}_{k,n}(X)) \to \pi_1(\mathcal{S}_{k,n}(X)) \]
 induced by the framing forgetting map $\mathcal{F}$, giving a {\em handle attachment homomorphism}
 \[ \mathcal{H}_{k,n} : \mathcal{F}_*(\pi_1(\mathcal{FS}_{k,n}(X))) \to \pi_0(\Diff^+(X)). \]
 Thus if $\mathcal{F}_*$ is surjective we have a homomorphism
 \[ \mathcal{H}_{k,n} : \pi_1(\mathcal{S}_{k,n}(X)) \to \pi_0(\Diff^+(X)) \]
 and $\imath_*(\pi_1(\mathcal{S}^0_{k,n}(X)))$ is contained in the kernel of $\mathcal{H}_{k,n}$. In the limit we get
 \[ \mathcal{H}_{k,\infty} : \pi_1(\mathcal{S}_{k,\infty}(X)) \to \pi_0(\Diff^+(X)) \]
 with $\imath_*(\pi_1(\mathcal{S}^0_{k,\infty}(X)))$ contained in the kernel.

\end{lemma}

\begin{proof}
Palais~\cite{Palais} shows that restriction maps such as $\mathcal{F}$ are locally trivial and thus satisfy the homotopy lifting property. The fiber of $\mathcal{F}$ is the space of framings of a fixed $S^k$, i.e. (up to homotopy) maps from $S^k$ to $SO(m-k)$. Note that the fiber over the basepoint is actually a subspace of $\mathcal{FS}^0_{k,n}(X)$ and thus $\pi_1$ of the fiber lands in the kernel of $\mathcal{FH}_{k,n}$. As a consequence, even though 
\[ \mathcal{F}_* : \pi_1(\mathcal{FS}_{k,n}(X)) \to \pi_1(\mathcal{S}_{k,n}(X)) \]
may not be injective, if the fiber is not simply connected, we still have that $\mathcal{FH}_{k,n}$ induces a well-defined homomorphism $\mathcal{H}_{k,n}$ from the image $\mathcal{F}_*(\pi_1(\mathcal{FS}_{k,n}(X)))$ of $\mathcal{F}_*$ in $\pi_1(\mathcal{S}_{k,n}(X))$ to $\pi_0(\Diff^+(X))$. All of this also commutes with the inclusion maps from $n$ to $n+1$ giving the results for $\mathcal{H}_{k,\infty}$.
\end{proof}

Finally we return to the case of relevance to Theorem~\ref{T:H2infSurj}. Here we have $m=4$ and $k=2$ and the base manifold $X=S^4$.

\begin{lemma} \label{L:EverythingButSurjectivity}
 The induced map
 \[ \mathcal{F}_* : \pi_1(\mathcal{FS}_{2,n}(S^4)) \to \pi_1(\mathcal{S}_{2,n}(S^4)) \]
 is surjective. Thus we get a well-defined homomorphism
 \[ \mathcal{H}_{2,n} : \pi_1(\mathcal{S}_{2,n}(S^4)) \to \pi_0(\Diff^+(S^4)) \]
 with $\imath_*(\mathcal{S}^0_{2,n}(S^4))$ in the kernel, and a limit homomorphism
 \[ \mathcal{H}_{2,\infty} : \pi_1(\mathcal{S}_{2,\infty}(S^4)) \to \pi_0(\Diff^+(S^4)) \]
 with $\imath_*(\mathcal{S}^0_{2,\infty}(S^4))$ in the kernel.
\end{lemma}

\begin{proof}
The fibers of the framing forgetting map $\mathcal{F}$ are path-connected, i.e. a $2$--sphere in a $4$--manifold with self-intersection $0$ has only one framing up to isotopy, since $\pi_2(SO(2))=0$. Thus the long exact sequence of homotopy groups gives the desired surjectivity.
\end{proof}

\section{Surjectivity of $\mathcal{H}_{2,\infty}$} \label{S:Surjectivity}

In this section we will complete the proof of Theorem~\ref{T:H2infSurj}, by showing that $\mathcal{H}_{2,\infty} : \pi_1(\mathcal{S}_\infty) \to \pi_0(\Diff^+(S^4))$ is surjective.

We will use Cerf theoretic techniques, beginning with a pseudoisotopy. Recall that a pseudoisotopy from the identity diffeomorphism of a manifold $X$ to a diffeomorphism $\phi: X \to X$ is a diffeomorphism $\Phi: [0,1] \times X \to [0,1] \times X$ which restricts to $\{0\} \times X$ as the identity and to $\{1\} \times X$ as $\phi$.
\begin{lemma}
 Every orientation preserving self-diffeomorphism of $S^4$ is pseudoisotopic to the identity.
\end{lemma}

\begin{proof}
 Consider an orientation preserving diffeomorphism $\phi: S^4 \to S^4$. Let $f: S^5 \to \R$ be projection onto the last coordinate in $\R^6$, and for any interval $I \subset \R$, let $S^5_I = f^{-1}(I)$. Let $V$ be a smooth vector field on $S^5 \setminus \{(0,0,0,0,0,\pm 1)\}$ which is orthogonal to all level sets of $f$ and scaled so that $df(V)=1$. Let $X = S^5_{[-1,0]} \cup_\phi S^5_{[0,1]}$, where $\phi: \partial S^5_{[0,1]} = S^4 \to -S^4 = \partial S^5_{[-1,0]}$ is now seen as an orientation reversing gluing diffeomorphism. Arrange the gluing (i.e. the smooth structure on $X$) so that the vector field $V$ on the two halves of $X$ is still a smooth vector field on $X$, which we call $V_X$. Note that $X$ also inherits the Morse function $f$, which we label $f_X:X\to \R$, and we can use the same notation $X_I = f_X^{-1}(I) \subset X$. The point is that if $I \subset (-\infty,0]$ or if $I \subset [0,\infty)$ then $X_I = S^5_I$, i.e. they are actually equal sets, not just diffeomorphic manifolds.
 
 Now note that $X$ is homotopy equivalent to $S^5$ and therefore~\cite{KervaireMilnor,Levine} diffeomorphic to $S^5$. For some small $\epsilon>0$ we can assume that the diffemorphism $\Phi: S^5 \to X$ is the identity on $S^5_{[-1,-1+\epsilon]} = X_{[-1,-1+\epsilon]}$ and on $S^5_{[1-\epsilon,1]} = X_{[1-\epsilon,1]}$. Using flow along $V$ and $V_X$, respectively, and the standard identification of $\partial S^5_{[-1,-1+\epsilon]}$ with $S^4$, we can parametrize both $S^5_{[-1+\epsilon,1-\epsilon]}$ and $X_{[-1+\epsilon,1-\epsilon]}$ as $[-1+\epsilon,1-\epsilon] \times S^4$ and then $\Phi$ restricts to give a diffeomorphism from $[-1+\epsilon,1-\epsilon] \times S^4$ to itself. Furthermore, this map is the identity on $\{-1+\epsilon\} \times S^4$ and by continuity must equal $\phi$ on $\{1-\epsilon\} \times S^4$. After reparametrizing $[-1+\epsilon,1-\epsilon]$ as $[0,1]$ we get the desired pseudoisotopy.
 
\end{proof}

Given a pseudoisotopy $\Phi: [0,1] \times S^4 \to [0,1] \times S^4$ from the identity to $\phi: S^4 \to S^4$, let
\[ Z(\Phi) = [0,1] \times [0,1] \times S^4 / (1,y,p) \sim (0,\Phi(y,p)) \]
be the mapping torus for $\Phi$, interpreted as a bundle over $S^1$ and as a fiberwise cobordism from $S^1 \times S^4$ to $Y(\phi) = [0,1] \times S^4/(1,p) \sim (0,\phi(p))$. We use the variable $y$ for the second $[0,1]$ factor because later this will play the role of a ``height'' rather than a ``time parameter''. We will reserve $t$ for the first $[0,1]$ factor, which in fact is the base $S^1 = [0,1] / 1 \!\sim\! 0$ variable.
\begin{definition} \label{D:FiberwiseHandle}
A {\em fiberwise handle decomposition} of $Z(\Phi)$ is a decomposition of $Z(\Phi)$ into a collar neighborhood $[0,1] \times S^1 \times S^4$ of the bottom boundary $S^1 \times S^4$ and a collection of round handles $S^1 \times B^k \times B^{5-k}$, for various $k$, on each of which the bundle map $Z(\Phi) \to S^1$ pre-composed with the characteristic map of the handle agrees with the projection $S^1 \times B^k \times B^{5-k} \to S^1$. This induces an ordinary handle decompositon of each fiber $Z(\Phi)_t$ as a collar neighborhood $[0,1] \times \{t\} \times S^4$ of the bottom boundary $\{t\} \times S^4$ and a collection of handles $\{t\} \times B^k \times B^{5-k}$.
\end{definition}

Once we establish the following proposition it will not take much work to translate the result into our main surjectivity of $\mathcal{H}_{2,\infty}$ result.

\begin{proposition} \label{P:AllTheCerfTheory}
 The cobordism $Z(\Phi)$ has a fiberwise handle decomposition with the following properties, for some $n \in \N$:
 \begin{enumerate} 
  \item For each $t \in S^1$, the induced handle decomposition of the $5$--dimensional fiber $Z(\Phi)_t$ over $t$ involves exactly $n$ $2$--handles and $n$ $3$--handles.
  \item There is a fixed embedding 
  \[ \alpha: \amalg^n(S^1 \times B^3) \into S^4 \]
  such that the $n$ $2$--handles for each $Z(\Phi)_t$ are attached to $\{t\} \times [0,1] \times S^4$ along $\alpha$, as an embedding into $\{t\} \times \{1\} \times S^4$. In other words, the attaching maps for the $2$--handle do not vary with $t$. As a result, for each $t$ the $4$--manifold immediately above these $2$--handles is canonically identified with $\#^n (S^2 \times S^2)$, and putting all the fibers together, the $5$--manifold immediately above these round $2$--handles is canonically identified with $S^1 \times \#^n (S^2 \times S^2)$.
  \item For each $t$, the attaching $2$--spheres for the $n$ $3$--handles of $Z(\Phi)_t$ are an embedding 
  \[ \beta_t: \amalg^n S^2 \into \#^n (S^2 \times S^2)^\dagger \subset \#^n (S^2 \times S^2) \] 
  which varies smoothly in $t$, forming a loop of embeddings of $\amalg^n S^2$ into $\#^n (S^2 \times S^2)^\dagger$ starting and ending at the standard embedding $\amalg^n (S^2 \times \{p\})$. In other words, the $3$--handles are attached along a based loop $\beta_t$ in $\mathcal{S}_{2,n}(S^4)$.
 \end{enumerate}

\end{proposition}
The ``canonical identification'' mentioned in item (2) in the statement can be made completely explicit, if desired, by first fixing the standard handle decomposition of $S^2 \times B^3$ with a single $0$--handle and a single $2$--handle, using this to fix the standard handle decomposition of $\natural^n S^2 \times B^3$ with one $0$--handle and $n$ $2$--handles, and then removing a small neighborhood of the index $0$ critical point.

\begin{proof}[Proof of Proposition~\ref{P:AllTheCerfTheory}]

 Let $f_0: [0,1] \times S^4 \to [0,1]$ be projection onto the first factor, i.e. $f_0(y,p)=y$, let $V_0 = \partial_y$ be the unit vector field on $[0,1] \times S^4$ in the $[0,1]$ direction, let $(f_1,V_1) = \Phi^*(f_0,V_0) = (f_0 \circ \Phi, D\Phi^{-1}(V_0))$, and let $(f_t,V_t)$ be a generic homotopy of functions with gradient-like vector fields from $(f_0,V_0)$ to $(f_1,V_1)$. 
 
 To preempt potential confusion, note that that $[0,1]$ factor with coordinate $y$ in the domain $[0,1] \times S^4$ of each function $f_t$ should not be confused with the $[0,1]$ parameter with coordinate $t$. Furthermore, each function $f_t$ maps to $[0,1]$, and this $[0,1]$ target space, on which we will use the variable $z$, should also not be confused with either of the other two $[0,1]$'s.
 
 Hatcher and Wagoner (Chapter~VI, Proposition~3, page~214 of~\cite{HatcherWagoner}) show that such a family of functions $f_t$ with gradient-like vector fields $V_t$ can be homotoped rel $t \in \{0,1\}$ so as to arrange the following properties (see Figure~\ref{F:Nice23Cerf}):
 \begin{figure}
  \labellist
  \small\hair 2pt
  \pinlabel $2$ [tr] at 45 20 
  \pinlabel $2$ [tr] at 49 24 
  \pinlabel $2$ [tr] at 56 31 
  \pinlabel $3$ [br] at 45 67 
  \pinlabel $3$ [br] at 49 63 
  \pinlabel $3$ [br] at 56 55 
  \endlabellist
  \centering
  \includegraphics[width=10cm]{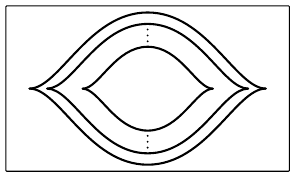}
  \caption{A nice Cerf graphic with only critical points of index $2$ and $3$.}
  \label{F:Nice23Cerf}
\end{figure}
 \begin{itemize}
  \item For each fixed value of the parameter $t \in [0,1]$:
  \begin{itemize}
   \item The only Morse critical points of $f_t$ are critical points of index $2$ and $3$.
   \item All critical points of $f_t$ map to distinct critical values. In other words, if $p$ and $q$ are distinct critical points of $f_t$ with critical values $z_p = f_t(p)$ and $z_q = f_t(q)$, then $z_p \neq z_q$.
   \item All critical points of index $2$ are below all critical points of index $3$. In other words, if $p$ is a critical point of index $2$ for $f_t$ with critical value $z_p = f_t(p)$ and $q$ is a critical point of index $3$ with critical value $z_q = f_t(q)$ then $z_p < z_q$.
   \item There are no handle slides. In other words, none of the flow lines of the vector field $V_t$ connect Morse critical points of the same index.
   \item At the moments of birth and death, the handle pair dying or being born does not run over any other handles. In other words, none of the flow lines of $V_t$ connect a non-Morse birth/death critical point to any other critical point.
   \item The function $f_t$ has at most one non-Morse birth/death critical point.
  \end{itemize}
  \item All births of cancelling pairs of critical points happen before all deaths of cancelling pairs. In other words, if $f_{t_0}$ has a non-Morse birth critical point and $f_{t_1}$ has a non-Morse death critical point, then $t_0 < t_1$.
 \end{itemize}
 
 Recall that $Z(\Phi) = [0,1] \times [0,1] \times X / (1,y,p) \sim (0,\Phi(y,p))$. In this context, let $Z(\Phi)_t = \{t\} \times [0,1] \times X$ be the fiber over $t \in S^1 = [0,1] / 1 \!\sim\! 0$. 
 Since $(f_1,V_1) = \Phi^* (f_0,V_0)$, we get a function $F: Z(\Phi) \to [0,1]$ and vector field $V$ on $Z(\Phi)$ such that $(F,V)|_{Z(\Phi)_t} = (f_t,V_t)$. Each fiber $Z(\Phi)_t$ of $Z(\Phi) \to S^1$ gets a handle decomposition from $(f_t,V_t)$ (allowing for birth and death handle decompositions) but it will take some work to arrange that these satisfy the remaining properties as stated in Proposition~\ref{P:AllTheCerfTheory}. First let us characterize each fiber $Z(\Phi)_t$, with its given handle decomposition, as much as possible in terms of handle attaching data in $\{t\} \times S^4 \subset S^1 \times S^4$.
 
 Suppose that the births in our Morse functions $f_t$ occur at times $0 < 1/8 < t_1 < t_2 < \ldots < t_n < 1/4$ and that the deaths occur at times $ 3/4 < t'_n < \ldots < t'_2 < t'_1 < 7/8 < 1$. We will show how each cobordism $Z(\Phi)_t$ can be described as built from $\{t\} \times [0,1] \times S^4$ according to handle attaching data. 
 
 To do this we establish a few notational conventions. First, all of our embeddings of spheres and disks of various dimensions are framed embeddings, but to keep the terminology minimal we will sometimes suppress mention of framings. Second, given a framed embedding $\epsilon$ of a sphere in an $m$--manifold $X$, let $X(\epsilon)$ denote the result of surgering $X$ along $\epsilon$. Third, a framed embedding $\delta$ of a disk $D^k$ into an $n$--manifold $X$ gives us several auxiliary pieces of information:
 \begin{itemize}
  \item We get a framed embedding $\alpha$ of $S^{k-1} = \partial D^k$ into $X$.
  \item We get the surgered manifold $X(\alpha)$.
  \item In $X(\alpha) = X \setminus \alpha(S^{k-1} \times \mathring{B}^{m-k+1}) \cup_\alpha (B^k \times S^{m-k})$, there is a natural framed embedding $\beta$ of $S^k$ into $X(\alpha)$ which coincides with $\delta$ away from the surgery, i.e. in $X \setminus \alpha(S^{k-1} \times \mathring{B}^{m-k+1})$, and which coincides with the core $B^k \times \{0\}$ of the surgery inside the surgered region $B^k \times S^{m-k}$.
  \item This embedding $\beta$ of $S^k$ also comes with a dual $(m-k)$--sphere $\beta^*$ which is the meridian to $\alpha$, i.e. $\alpha(\{p\} \times S^{m-k})$, so that $\beta$ and $\beta^*$ intersect transversely at one point in $X(\alpha)$.
 \end{itemize}
 This is the basic model for the result of attaching a cancelling pair of a $(m+1)$--dimensional $k$--handle and a $(m+1)$--dimensional $(k+1)$--handle, with the attaching data for the pair being described entirely in $X$ by the framed disk $\delta$. Thus we also see that $X(\alpha)(\beta)$ is canonically identified with $X$. In this proof we will work with the case $k=2$ and $k+1=3$, but later in the paper we will be interested in the case $k=1$ and $k+1=2$.
 
 The notation introduced above also makes sense when the $\delta$'s, $\alpha$'s or $\beta$'s are embeddings of disjoint unions of disks or spheres. Using this, we now describe the form of the explicit $t$--varying handle attaching data that gives the fiberwise construction of $Z(\Phi)$, i.e. the data that shows how to construct each $Z(\Phi)_t$ starting with $\{t\} \times [0,1] \times S^4$ and attaching various $5$--dimensional handles.
 \begin{itemize}
  \item For $0 \leq t < t_1$, no handles are attached, i.e. $Z(\Phi)_t = \{t\} \times [0,1] \times S^4$.
  \item For $t=t_1$, there is a framed embedding $\delta_1$ of a disk $D^2$ into $S^4$ such that $Z(\Phi)_{t_1}$ is built from $\{t_1\} \times [0,1] \times S^4$ by attaching a cancelling pair of a $5$--dimensional $2$--handle and $3$--handle, with the $2$--handle attached along the framed embedding $\alpha^1_{t_1} = \delta_1|_{S^1}$, and with the $3$--handle attached along the resulting framed embedding $\beta^1_{t_1}$ of $S^2$ into $S^4(\alpha^1_{t_1})$.
  \item For $t_1 \leq t \leq t_2$, we have a $t$--varying framed embedding $\alpha_t = \alpha^1_t$ of $S^1$ into $S^4$ and a $t$--varying framed embedding $\beta_t = \beta^1_t$ of $S^2$ into $S^4(\alpha_t)$. For $t_1 \leq t < t_2$, $Z(\Phi)_t$ is built from $\{t\} \times [0,1] \times S^4$ by attaching a $2$--handle along $\alpha_t$ and then a $3$--handle along $\beta_t$. At $t=t_1$, the $\alpha_t$ and $\beta_t$ agree with the $\alpha^1_t$ and $\beta^1_t$ from the previous point.
  \item For $t=t_2$, there is a framed embedding $\delta_2$ of $D^2$ into $S^4$ disjoint from the images of $\alpha^1_{t_2}$ and $\beta^1_{t_2}$, such that $Z(\Phi)_{t_2}$ is built from $\{t_2\} \times [0,1] \times S^4$ by attaching
  \begin{itemize}
   \item first a $5$--dimensional $2$--handle along $\alpha^1_{t_2}$,
   \item then a pair of cancelling $2$-- and $3$--handles in which the $2$--handle is attached along $\alpha^2_{t_2} = \delta_2|_{S^1}$ and the $3$--handle is attached along the resulting framed $2$--sphere $\beta^2_{t_2}$ in $S^4(\alpha^1_{t_2})(\alpha^2_{t_2})$, and
   \item then a $3$--handle attached along $\beta^1_{t_2}$, which can be seen as a framed $2$--sphere in $S^4(\alpha^1_{t_2})$, in $S^4(\alpha^1_{t_2})(\alpha^2_{t_2})$ or in $S^4(\alpha^1_{t_2})(\alpha^2_{t_2})(\beta^2_{t_2}) \cong S^4(\alpha^1_{t_2})$.
  \end{itemize}
  \item For $t_2 \leq t \leq t_3$, we have a $t$--varying framed embedding $\alpha_t = \alpha^1_t \amalg \alpha^2_t$ of $S^1 \amalg S^1$ into $S^4$ and a $t$--varying framed embedding $\beta_t = \beta^1_t \amalg \beta^2_t$ of $S^2 \amalg S^2$ into $S^4(\alpha_t)$, agreeing with the $\alpha^1_{t_2}$, $\alpha^2_{t_2}$, $\beta^2_{t_2}$ and $\beta^1_{t_2}$ of the preceding point when $t=t_2$, so that $Z(\Phi)_t$ is built from $\{t\} \times [0,1] \times S^4$ by attaching $2$--handles along $\alpha_t$ and then $3$--handles along $\beta_t$.
  \item This process continues with each birth at time $t_i$ governed by a new framed disk $\delta_i$, generating a new framed $S^1$, $\alpha^i_{t_i}$, and a new framed $S^2$, $\beta^i_{t_i}$, which then join the previous framed spheres to create $\alpha_t = \alpha^1_t \amalg \ldots \amalg \alpha^i_t$ and $\beta_t = \beta^1_t \amalg \ldots \amalg \beta^i_t$, which are the attaching spheres for $2$-- and $3$--handles for $t_t \leq t < t_{i+1}$.
  \item Reversing time we see the deaths governed by (most likely quite different) disks $\delta_n', \ldots, \delta_1'$ and the same pattern of framed $S^1$'s and $S^2$'s in between these times.
  \item For $t_n \leq t \leq t_n'$, there is a $t$--parametrized family $\alpha_t$ of framed embeddings of $\amalg^n S^1$ into $S^4$ and a $t$--parametrized family $\beta_t$ of framed embeddings of $\amalg^n S^2$ into $S^4(\alpha_t)$ which constitute the attaching data for the $n$ $2$--handles and $n$ $3$--handles used to construct each $Z(\Phi)_t$ in this range.
 \end{itemize}
We will now improve the format of this data somewhat. First, we can arrange that for some small $\epsilon>0$, on the time interval $t_1 \leq t \leq t_1 + \epsilon$ the embeddings $\alpha_t$ and $\beta_t$ are independent of $t$, i.e. the first framed circle and sphere do not move for a short time after their birth. Next, since a birth happens at a point, we can make the second birth happen at an earlier time so that in fact $t_1 < t_2 < t_1+\epsilon$. Repeating this, and doing the same in reverse with the deaths, we can assume that on the whole interval $t_1 \leq t \leq t_n$ and on the whole interval $t_n' \leq t \leq t_1'$, the $\alpha_t$ and $\beta_t$ are independent of $t$ except for the fact that at each $t_i$ a new $\alpha^i_t$ and $\beta^i_t$ is added to the mix (and ditto for the deaths).
  
Thus now the governing data for the constructions of each $Z(\Phi)_t$ can be more succinctly described by the following data:
\begin{itemize}
 \item A framed embedding $\delta = \delta_1 \amalg \ldots \amalg \delta_n$ of $\amalg^n D^2$ into $S^4$ (for the births), defining framed embeddings $\alpha$ of $\amalg^n S^1$ into $S^4$ and $\beta$ of $\amalg^n S^2$ into $S^4(\alpha)$. We will call these the ``birth disks''.
 \item A framed embedding $\delta' = \delta_1' \amalg \ldots \amalg \delta_n'$ of $\amalg^n D^2$ into $S^4$ (for the deaths), defining framed embeddings $\alpha'$ of $\amalg^n S^1$ into $S^4$ and $\beta'$ of $\amalg^n S^2$ into $S^4(\alpha')$. We will call these the ``death disks''.
 \item A $t$--parameterized family, for $t \in [1/4,3/4]$, of framed embeddings $\alpha_t = \alpha^1_t \amalg \ldots \amalg \alpha^n_t$ of $\amalg^n S^1$ into $S^4$, with $\alpha_{1/4} = \alpha$ and $\alpha_{3/4} = \alpha'$.
 \item A $t$--parameterized family, for $t \in [1/4,3/4]$, of framed embeddings $\beta_t = \beta^1_t \amalg \ldots \amalg \beta^n_t$ of $\amalg^n S^2$ into $S^4(\alpha_t)$, with $\beta_{1/4} = \beta$ and $\beta_{3/4} = \beta'$.
\end{itemize}
In the time interval $t \in [1/8,1/4]$, the pairs of $2$--handles and $3$--handles are born one after another but their attaching circles and spheres do not move after birth, until all pairs are born and then the motion starts at $t=1/4$. Similarly all motion stops at $t=3/4$ and then the pairs die one after another in the time interval $t \in [3/4,3/8]$.

Our next step is to arrange that $\alpha_t = \alpha$ for all $t \in [1/4,3/4]$, using a standard ``run everything off the end'' argument. We give this argument in careful detail here, and then we appeal to a similar argument without as much detail in the next stage of the proof. Let $\tilde{\alpha}_t$ be a family of embeddings of $\amalg^n S^1 \into S^4$ defined for $t \in [0,1]$ as follows:
\begin{itemize}
 \item For $t \in [0,1/4]$, $\tilde{\alpha}_t = \alpha_{1/4} = \alpha$.
 \item For $t \in [1/4,3/4]$, $\tilde{\alpha}_t = \alpha_t$.
 \item For $t \in [3/4,7/8]$, $\tilde{\alpha}_t = \alpha_{3/4} = \alpha'$.
 \item For $t \in [7/8,1]$, $\tilde{\alpha}_t = \tilde{\alpha}_{7-7t}$; this simply means that on $[7/8,1]$, the family $\tilde{\alpha}_t$ is the same as $\tilde{\alpha}_t$ on $[0,7/8]$, but sped up and run backwards so that $\tilde{\alpha}_1 = \tilde{\alpha}_0$.
\end{itemize}
Use the isotopy extension theorem to produce a family of diffeomorphisms $\psi_t: S^4 \to S^4$, for $t \in [0,1]$, such that $\psi_t \circ \tilde{\alpha}_t = \tilde{\alpha}_0 = \alpha$ for all $t \in [0,1]$ and satisfying:
\begin{itemize}
 \item For $t \in [0,1/4]$, $\psi_t = \id$.
 \item For $t \in [3/4,7/8]$, $\psi_t = \psi_{3/4}$, i.e. $\psi_t$ does not vary with $t$ in this time interval.
 \item For $t \in [7/8,1]$, $\psi_t = \psi_{7-7t}$ and in particular,
 \item $\psi_1 = \id$.
\end{itemize}

Since $\psi_t$, as a loop of diffeomorphisms of $S^4$ starting and ending at the ending at the identity, simply traverses a path from $\psi_0 = \id$ to $\psi_{7/8}$ and then follows exactly the same path backwards to $\id$, there exists a $2$--parameter family of diffeomorphisms $\phi_{s,t}: S^4 \to S^4$ such that $\phi_{0,t} = \id$, $\phi_{1,t} = \psi_t$ and $\phi_{s,0} = \phi_{s,1} = \id$ for all $s \in [0,1]$. We can use this to produce a $2$--parameter family of gradient-like vector fields $V_{s,t}$ for the $1$--parameter family of Morse functions $f_t$ on $[0,1] \times S^4$ satisfying:
\begin{itemize}
 \item For each $(s,t) \in [0,1] \times [0,1]$, $V_{s,t}$ is gradient-like for $f_t$.
 \item For each $t \in [0,1]$, $V_{0,t} = V_t$.
 \item For each $s \in [0,1]$, $V_{s,0} = V_0$ and $V_{s,1} = V_1$.
 \item There are values $0<y_0<y_1<1$ such that, for all $t \in [0,1]$, there are no critical values of $f_t$ in $[0,y_1]$ and such that, for all $(s,t) \in [0,1] \times [0,1]$, $V_{s,t}$ agrees with $V_t$ outside $f_t^{-1}([y_0,y_1])$.
 \item Using flow along $V_t$ to identify $f_t^{-1}(y_0)$ and $f_t^{-1}(y_1)$ with $S^4$, we have that for all $(s,t) \in [0,1] \times [0,1]$ downward flow along $V_{s,t}$ from $f^{-1}(y_1)$ to $f^{-1}(y_0)$ is the diffeomorphism $\phi_{s,t}$.
\end{itemize}

Then we see that, for the $1$--parameter family $(f_t,V_{1,t})$, the attaching circles for the $2$--handles are precisely $\phi_{1,t} \circ \alpha_t = \psi_t \circ \alpha_t = \alpha$. Now we use $V_{1,t}$ as our new gradient-like vector field and we have succeeded in making sure that the $2$--handle attaching maps do not move with time, i.e. $\alpha_t = \alpha$ for all $t \in [1/4,3/4]$. The ``trick'' used is that in the time interval $t \in [7/8,1]$ on which we run everything backwards, there are no critical points, so we do not have to worry about what running things backwards does to the attaching data for handles. Also note that, because $\psi_t = \id$ does not vary with $t$ on the intervals $t \in [1/8,1/4]$ and $t \in [3/4,7/8]$ when the births and deaths occur, we still have the property that the attaching circles and spheres for the $2$-- and $3$--handles do not move after their births until $t=1/4$ and that all motion stops at $t=3/4$, after which the cancelling handle pairs die one by one. The only differences are that the attaching spheres $\alpha_t$ for the $2$--handles do not vary at all with $t$, over the entire time interval $t \in [1/4,3/4]$. We have preserved our orginal configuration of birth disks $\delta$, but our death disks $\delta'$ have now changed, but they still do not vary with $t$ during the death interval $t \in [3/4,7/8]$. Although $\delta'$ has now changed, we will relabel the new death disks as $\delta'$.

Now we would like to arrange that $\delta'=\delta$. This is easy to arrange if we allow the boundaries of the disks to move, but we have already arranged that $\partial \delta' = \partial \delta = \alpha$ and we want to preserve this feature. We do this as an application of Gabai's $4$--dimensional lightbulb theorem, in particular the multiple spheres version Theorem~10.1 in~\cite{GabaiLightBulb}.

Recall from our discussion of handle and surgery terminology earlier that, given the $2$--$3$--birth disks $\delta$, we immediately get the attaching $2$--spheres for the $3$--handles $\beta$ and their dual $2$--spheres $\beta^*$. The same holds for the death disks $\delta'$, giving attaching $2$--spheres $\beta'$ with dual $2$--spheres $\beta'^*$. But since $\partial \delta = \partial \delta'$ and since these dual spheres can be taken to the meridians to the boundaries of these disks, in fact we now have arranged that $\beta'^* = \beta^*$. Thus, thinking about our $1$--parameter family $\beta_t$ of attaching spheres for the $3$--handles for $t \in [1/4,3/4]$, we know that $\beta_{1/4}$ and $\beta_{3/4}$ have a common set of dual spheres $\beta^*$ and in fact $\beta_{1/4}$ and $\beta_{3/4}$ agree in a neighborhood of $\beta^*$. Since $\beta_{3/4}$ is isotopic to $\beta_{1/4}$ (although this isotopy does not preserve geometric duality with $\beta^*$), they are homotopic and thus the lightbulb theorem applies. The conclusion of the theorem is that $\beta_{3/4}$ is isotopic to $\beta_{1/4}$ via an isotopy that fixes a neighborhood of $\beta^*$. This isotopy happens in the surgered manifold $S^4(\alpha)$, but since surgering $S^4(\alpha)$ along $\beta^*$ recovers $S^4$, the fact that the isotopy fixes a neighborhood of $\beta^*$ means precisely that we get an isotopy rel. boundary of $\delta_{3/4} = \delta'$ to $\delta_{1/4} = \delta'$ in $S^4$.

Let us assume that $\beta_t = \beta_{3/4}$ for all $t \in [5/8,3/4]$. We can use the isotopy from the preceding paragraph and a ``run everything off the end argument'' similar to the $\alpha_t$ argument above to modify our family of gradient-like vector fields $V_t$ so as to finally arrange that $\delta' = \delta$ as desired. More precisely, let $\tilde{\delta}_t$, for $t \in [5/8,3/4]$, be a family of embeddings of $\amalg^n D^2$ into $S^4$ such that $\tilde{\delta}_{5/8} = \delta'$ and $\tilde{\delta}_{3/4} = \delta$. Extend this to an ambient isotopy $\psi_t: S^4 \to S^4$ for $t \in [0,1]$ such that:
\begin{itemize}
 \item For all $t \in [0,5/8]$, $\psi_t = \id$. 
 \item For all $t \in [5/8,3/4]$, $\psi_t \circ \delta' = \tilde{\delta}_t$.
 \item For all $t \in [3/4,7/8]$, $\psi_t = \psi_{3/4} = \psi_{7/8}$.
 \item For $t \in [7/8,1]$, $\psi_t = \psi_{13/8 - t}$, i.e $\psi_t$ is the same as $\psi_t$ on $[5/8,3/4]$ but run backwards in time.
\end{itemize}
As before, $\psi_t$ is a null homotopic loop of diffeomorphisms of $S^4$ based at $\id$, so we let $\phi_{s,t}$ be a homotopy rel $t \in \{0,1\}$ from $\phi_{0,t} = \id$ to $\phi_{1,t} = \psi_t$, and then use this to modify our gradient like vector field $V_t$ rel $t \in \{0,1\}$. The upshot is that we are able to modify our attaching maps by $\psi_t$ and thus arrange that our death disks $\delta'$ become equal to our birth disks $\delta$.

To complete the proof, since the birth and death disks are now equal, we can move the births closer and closer to $t=0$ and move the deaths closer and closer to $t=1$ and then merge the deaths with the births at $t = 0 = 1 \in S^1 = [0,1]/1\!\sim\!0$. Now we have a $t$--parameterized family $(f_t,V_t)$ giving a fiberwise handle decomposition of $Z(\Phi)$ satisfying almost everything as advertised in the statement of the proposition. The last detail to arrange is that the loop of embeddings $\beta_t$ of $\amalg^n S^2 \into \#^n(S^2 \times S^2)$ all miss a single fixed point in $\#^n(S^2 \times S^2)$. This is done simply by observing that the trace of the family is a smooth map of a $3$--manifold $S^1 \times \amalg^n S^2$ into a $4$--manifold which, by Sard's theorem, must miss a point.

\end{proof}

\begin{proof}[Proof of Theorem~\ref{T:H2infSurj}]
In Section~\ref{S:LoopsToDiffs} we have already constructed the homomorphisms $\mathcal{H}_{2,n}$, shown that it commutes with $j_*$, and thus constructed $\mathcal{H}_{2,\infty}$, and we have shown that $\imath_*(\pi_1(\mathcal{S}^0_{2,\infty}(S^4)))$ lies in the kernel of $\mathcal{H}_{2,\infty}$; this was all summarized as Lemma~\ref{L:EverythingButSurjectivity}. It remains to prove that $\mathcal{H}_{2,\infty}$ is surjective.

Since $\mathcal{H}_{2,\infty}$ is a limit of maps $\mathcal{H}_{2,n}$, what we are really trying to prove is the following:

For any orientation preserving diffeomorphism $\phi: S^4 \to S^4$, there is some $n \in \N$ and some loop $\beta_t$ in $\mathcal{S}_n$ based at the basepoint $\amalg^n (S^2 \times \{p\})$, such that $\mathcal{H}_{2,n}([\beta_t]) = [\phi] \in \pi_0(\Diff^+(S^4))$. Choose a pseudoisotopy $\Phi$ from $\id$ to $\phi$, build the bundle of cobordisms $Z(\Phi)$ and apply Proposition~\ref{P:AllTheCerfTheory}. This gives us exactly the loop $\beta_t$, and we see that $Z(\Phi)$ is exactly the result of parameterized handle attachment as in Definition~\ref{D:FHn}.
\end{proof}

\section{Turning $(2,3)$--handle pairs into $(1,2)$--handle pairs} \label{S:23to12}

We now need to work toward the connection with Montesinos twins and the proof of Theorem~\ref{T:TwinsGenerate}, that twists along Montesinos twins generate the subgroup of $\pi_0(\Diff^+(S^4))$ corresponding to loops of $2$--spheres which remain disjoint from parallel copies of the basepoint $2$--spheres. More precisely, recall that $\widehat{\mathcal{S}}_{2,n}(S^4)$ is the space of embeddings of $\amalg^n S^2$ into $\#^n (S^2 \times S^2) ^\dagger$ which remain disjoint from $\amalg^n (S^2 \times \{p'\})$ for some fixed $p' \in S^2$, with basepoint being $\amalg^n (S^2 \times \{p\})$ for $p \neq p'$. We want to study the subgroup $\mathcal{H}_{2,\infty}(\imath_*(\pi_1(\widehat{\mathcal{S}}_{2,\infty}(S^4))))$.

As in the title of this section, we will frequently refer to $(k,k+1)$--handle pairs.
\begin{definition} \label{D:kkplus1pair}
 A {\em $(k,k+1)$--handle pair} is a pair of handles, one being a $k$--handle and the other being a $(k+1)$--handle, such that the attaching sphere for the $(k+1)$--handle can be isotoped in the level above the $k$--handle to intersect the belt sphere for the $k$--handle transversely at one point. A $(k,k+1)$--handle pair is {\em in cancelling position} if the attaching sphere and the belt sphere intersect transversely at exactly one point, without performing an isotopy first.
\end{definition}
In other words, for a general $(k,k+1)$--handle pair, after an isotopy of the attaching sphere of the $(k+1)$--handle, the pair of handles can be cancelled, but they are not necessarily in cancelling position before this isotopy is performed.
 
To relate $\mathcal{H}_{2,\infty}(\imath_*(\pi_1(\widehat{\mathcal{S}}_{2,\infty}(S^4))))$ to Montesinos twins we will first need to relate it to families of $(1,2)$--handle pairs coming from loops of framed circles in $\#^n (S^1 \times S^3)$. In particular, using the notation from Section~\ref{S:LoopsToDiffs}, in this section we will prove:
\begin{theorem} \label{T:23to12}
 For any $n$, $\mathcal{H}_{2,n}(\imath_*(\pi_1(\widehat{\mathcal{S}}_{2,n}(S^4)))) = \mathcal{FH}_{1,n}(\pi_1(\mathcal{FS}_{1,n}(S^4)))$.
\end{theorem}
This means that any isotopy class of diffeomorphisms of $S^4$ that can be realized by a family of cobordisms built with $n$ $(2,3)$--handle pairs governed by a loop in $\widehat{\mathcal{S}}_{2,n}(S^4)$ can also be realized by a family of cobordisms built with $n$ $(1,2)$--pairs governed by a loop in $\mathcal{FS}_{1,n}(S^4)$, which is the space of framed embeddings of a disjoint union of $n$ circles in $\#^n (S^1 \times S^3)^\dagger$.

The essential idea is that, when the $3$--handle attaching maps lie in $\widehat{\mathcal{S}}_{2,n}(S^4)$, we can use a $5$--dimensional analog of the $4$--dimensional Kirby calculus ``dotted circle'' notation to represent our $2$--handles as dotted circles and our $3$--handles as $2$--spheres moving in the complement of these dotted circles. Then we can do the $5$--dimensional analog of the $4$--dimensional trick of ``switching dots and zeros'' to turn the dotted circles into attaching circles for $2$--handles and to turn the $2$--spheres into dotted $2$--spheres which represent $1$--handles. Note that in general this process changes the underlying cobordism. In addition to the subtleties involved in working in dimension $5$ rather than dimension $4$, we have the added complication that we need to do this in a $1$--parameter family. The author would like to thank Peter Teichner and Danica Kosanovic for first making him aware of the potential of $5$--dimensional dotted circle and sphere notation when discussing Watanabe's work~\cite{Watanabe}.

For the reader who is not comfortable with $4$--dimensional dotted circle notation, the discussion by Gompf and Stipsicz in Section~5.4 of~\cite{GompfStipsicz} should be helpful, with many more details than given here and with helpful illustrations. In general dimensions, a trivial $n$--dimensional $k$--handle is a $k$--handle which could be cancelled by a $(k+1)$--handle, in which case one way to attach a trivial $k$--handle to a $n$--manifold $X$ is to attach a cancelling $(k,k+1)$--handle pair to $X$ and then remove a neighborhood of the $B^{n-k-1}$ co-core of the $(k+1)$--handle. If one then cancels the $(k,k+1)$--handle pair, one is simply removing a regular neighborhood of a properly embedded boundary parallel $B^{n-k-1}$ from $X$, which can be indicated as a dotted $S^{n-k-2}$ in $\partial X$. Looking at the effect on $\partial X$, note that adding a trivial $k$--handle modifies $\partial X$ by a connected sum with $S^k \times S^{n-k-1}$, and that this can be also achieved by adding a trivial $(n-k-1)$--handle, so that from the point of view of the boundary the dotted $S^{n-k-2}$ can either be interpreted with its dot as describing a $k$--handle or simply as the attaching sphere for a $(n-k-1)$--handle. 

In the standard $4$--dimensional setting, $n=4$ and $k=1$, so that we have dotted circles. In our setting we will work with two cases, one where $n=5$ and $k=2$, giving dotted circles, and the other where $n=5$ and $k=1$, giving dotted $2$--spheres. We now describe these cases, and the $4$--dimensional case, more explicitly.

In dimension $4$, a dotted circle $\kappa: S^1 \into S^3$ in a Kirby diagram needs to be an unknot, and then one chooses a disk $\overline{\kappa}: D^2  \into S^3$ with $\partial \overline{\kappa} = \kappa$, pushes the interior of $\overline{\kappa}$ into a collar neighborhood $(-\epsilon,0] \times S^3$ of $S^3$ and then removes a regular neighborhood of $\kappa$. This has the same result as attaching a $4$--dimensional $1$--handle along an embedding $\zeta: S^0 \into S^3$ where $\zeta = \partial \overline{\zeta}$ for an embedding $\overline{\zeta}: B^1 \into S^3$ which is dual to $\overline{\kappa}$, i.e. $\overline{\zeta}$ transversely intersects $\overline{\kappa}$ at one point in their interiors.

There are several important points to note about this construction:
\begin{enumerate}
 \item In principle one needs to specify the disk $\overline{\kappa}$, not just its boundary $\kappa$. However, in dimension $3$ there is a unique disk bounded by an unknot so one ignores this issue. In higher dimensional analogs, and especially when working in $1$--parameter families, we should really keep track of the analog of the disk, not just the boundaries.
 \item Suppose that a $1$--handle is described by a dotted circle $\kappa: S^1 \into S^3$ bounding a disk $\overline{\kappa}: D^2 \into S^3$, and that a $2$--handle is attached along a knot $K: S^1 \into S^3 \setminus \kappa(S^1)$. If $K$ intersects $\overline{\kappa}$ transversely once, then the $1$--handle and $2$--handle are in cancelling position and can be cancelled.
 \item The real advantage, in $4$--dimensional Kirby calculus, of using dotted circle notation is that the attaching knots for the $2$--handles can be drawn entirely in $S^3$ even though they might run over some $1$--handles, and the $1$-- and $2$--handle information is given by a single link in $S^3$ including both the dotted circles and the attaching knots for the $2$--handles. Each $1$--handle creates an $S^1 \times S^2$ summand in the boundary of the $4$--manifold one is building, and the reason that we can draw everything in $S^3$ is that the attaching knots for the $2$--handles can be assumed to miss an $S^1 \times \{p\}$ in each of these $S^1 \times S^2$ summands, for some $p \in S^2$. In a $1$--parameter family this may not be the case, which is why one of the Kirby calculus moves when using dotted circle notation involves ``sliding a $2$--handle over a dotted circle'' as in Figure~5.36 of~\cite{GompfStipsicz}.
 \item Not all $1$--handles can be represented by a dotted circle. The $1$--handle needs to be a {\em trivial $1$--handle}, which means that {\em it could be cancelled} by a $2$--handle, since removing a neighborhood of a disk bounded by the dotted circle can be seen as first attaching the $1$--handle, then attaching a cancelling $2$--handle, and then removing a neighborhood of the co-core of that $2$--handle. A trivial $1$--handle is nothing more than a $1$--handle with both feet on the same component of the manifold to which it is being attached, and which does not create a non-orientable manifold after being attached. Alternatively, this condition is that the $1$--handle is attached along an $S^0$ that bounds a $B^1$ with a framing which extends across the $B^1$.
 \item Given a dotted circle $\kappa$, one can either use the dotted circle notation and interpret this as a $1$--handle or one can attach a $2$--handle along $\kappa$ with framing $0$. These produce different $4$--manifolds but the new boundaries created are the same. In other words, from the point of view of the $3$--manifold boundary, putting a dot or a $0$ on an unknotted circle in a Kirby diagram describes the same resulting $3$--manifold.
\end{enumerate}

Now we spell out, with more detail, the $5$--dimensional analogs of these points for $5$--dimensional $1$--handles. 
\begin{enumerate}
 \item Let $W$ be a $5$--dimensional cobordism from $4$--manifold $X_0$ to $4$--manifold $X_1$, and let $W'$ be the result of attaching a $5$--dimensional $1$--handle to $W$ via a framed embedding $\zeta: S^0 \times B^4 \into X_1$ which extends to a framed embedding $\overline{\zeta}: B^1 \times B^3 \into X_1$. This is a cobordism from $X_0$ to $X'_1$. Let $\beta^\bullet: B^3 \into X^1$ be the restriction of $\overline{\zeta}$ to $\{0\} \times B^3$. In other words, $\beta^\bullet$ is a dual $B^3$ to the $B^1$ bounded by the attaching $0$--sphere of the $1$--handle. The superscript $\bullet$ indicates that this is a ``dotted $3$--ball''. Let $W''$ be the $5$--manifold obtained by pushing the interior of $\beta^\bullet$ into a collar neighborhood $(-\epsilon,1] \times X_1$ of $X_1$ in $W$ and then removing a regular neighborhood of this pushed-in $3$--ball. This is a cobordism from $X_0$ to $X''_1$. Then $X'_1$ and $X''_1$ are diffeomorphic and $W''$ and $W'$ are diffeomorphic cobordisms. These diffeomorphisms are canonical in the sense that if we do this in parameterized families and build fiberwise cobordisms, we have corresponding fiberwise diffeomorphisms. For this reason we abuse notation and declare that $W'=W''$ and $X'_1=X''_1$.
 \item Since everything described above happens in a ball one can see directly that $X'_1 \cong X_1 \# (S^1 \times S^3)$. In the dotted $B^3$ notation, the $S^1$ factor is a linking meridian to the dotted $2$--sphere $\partial \beta^\bullet$, and the dotted $B^3$ itself is one hemisphere of the $S^3$ factor.
 \item Using dotted $3$--ball notation, the difference between $X'_1$ and $X_1$ is the $S^1 \times B^3$ part of the boundary of the neighborhood $B^2 \times B^3$ of the pushed-in $B^3$ that is removed from $W$ to produce $W'$. This $S^1 \times B^3$ is thus a regular neighborhood of $S^1 \times \{p\}$ in the $S^1 \times S^3$ summand of $X'_1 = X_1 \# (S^1 \times S^3)$, for some fixed $p \in S^3$.
 \item From this we see that if a $5$--dimensional $2$--handle is attached to $W'$ along an embedding $\alpha: S^1 \into X'_1 \cong X_1 \# (S^1 \times S^3)$ which is disjoint from $S^1 \times \{p\}$ for some fixed $p \in S^3$, then using dotted $B^3$ notation the attaching $S^1$ for the $2$--handle can be draw entirely in $X_1$ along with the dotted $B^3$'s.
 \item In this case, the $S^1$'s for the $2$--handles and the boundaries of the dotted $B^3$'s for the $1$--handles are disjoint, but the $S^1$'s may intersect the interiors of the dotted $B^3$'s, indicating that a $2$--handle is running over a $1$--handle.
 \item If the attaching $S^1$ for a $2$--handle transversely intersects a dotted $B^3$ for a $1$--handle exactly once, then the $(1,2)$--handle pair is in cancelling position and can be cancelled.
 \item Given the dotted $3$--ball $\beta^\bullet$, we could attach a $3$--handle along the $2$--sphere $\partial \beta^\bullet$ instead of the $1$--handle construction above. This produces a different cobordism, but the boundary $4$--manifold is exactly the same, since both constructions change the boundary by surgery along $\partial \beta^\bullet$.
\end{enumerate}

And finally we can give the analogous statements for $5$--dimensional $2$--handles.
\begin{enumerate}
 \item Let $W$ be a $5$--dimensional cobordism from $4$--manifold $X_0$ to $4$--manifold $X_1$, and let $W'$ be the result of attaching a $5$--dimensional $2$--handle to $W$ via a framed embedding $\alpha: S^1 \times B^3 \into X_1$ which extends to a framed embedding $\overline{\alpha}: B^2 \times B^2 \into X_1$. This is a cobordism from $X_0$ to $X'_1$. Let $\gamma^\bullet: B^2 \into X^1$ be the restriction of $\overline{\alpha}$ to $\{0\} \times B^2$. In other words, $\gamma^\bullet$ is a dual ``dotted disk'' to the disk bounded by the attaching circle of the $2$--handle. Let $W''$ be the $5$--manifold obtained by pushing the interior of $\gamma^\bullet$ into a collar neighborhood $(-\epsilon,1] \times X_1$ of $X_1$ in $W$ and then removing a regular neighborhood of this pushed-in disk. This is a cobordism from $X_0$ to $X''_1$ and again $W''$ and $W'$ are canonically diffeomorphic, and we again declare that $W'=W''$ and $X'_1=X''_1$.
 \item Now we have that $X'_1 \cong X_1 \# (S^2 \times S^2)$ where the first $S^2$ factor is a linking meridian to the dotted circle $\partial \gamma^\bullet$, and the dotted disk itself is one hemisphere of the second $S^2$ factor.
 \item Using dotted disk notation, the difference between $X'_1$ and $X_1$ is the $S^2 \times B^2$ part of the boundary of the neighborhood $B^3 \times B^2$ of the pushed-in $B^2$ that is removed from $W$ to produce $W'$. This $S^2 \times B^2$ is thus a regular neighborhood of $S^2 \times \{2\}$ in the $S^2 \times S^2$ summand of $X'_1 = X_1 \# (S^2 \times S^2)$, for some fixed $p \in S^2$.
 \item From this we see that if a $5$--dimensional $3$--handle is attached to $W'$ along an embedding $\beta: S^2 \into X'_1 \cong X_1 \# (S^2 \times S^2)$ which is disjoint from $S^2 \times \{p\}$ for some fixed $p \in S^2$, then using dotted disk notation the attaching $S^2$ for the $3$--handle can be draw entirely in $X_1$ along with the dotted disks.
 \item In this case, the $S^2$'s for the $3$--handles and the boundaries of the dotted disks for the $2$--handles are disjoint, but the $S^2$'s may intersect the interiors of the dotted disks, indicating that a $3$--handle is running over a $2$--handle.
 \item If the attaching $S^2$ for a $3$--handle transversely intersects a dotted disk for a $2$--handle exactly once, then the $(2,3)$--handle pair is in cancelling position and can be cancelled.
 \item Given the dotted disk $\gamma^\bullet$, we could attach a $2$--handle along the circle $\partial \gamma^\bullet$ instead of the $2$--handle construction above. This produces a (potentially) different cobordism, but the boundary $4$--manifold is again exactly the same, since both constructions change the boundary by surgery along $\partial \gamma^\bullet$.
\end{enumerate}

One might wonder whether we really need to keep track of the dotted $B^3$'s and disks or whether, as in the $4$--dimensional setting, one can just track the boundary spheres and circles. Budney and Gabai have shown~\cite{BudneyGabai} that unknotted $2$--spheres in $S^4$ can bound ``knotted'' $3$--balls, and of course, although all $S^1$'s in $S^4$ are unknotted, $2$--knots can be tied into any spanning disk for such an $S^1$. What really matters is whether these spanning disks and balls are isotopic in dimension $5$, and we leave this as an interesting question; the Budney-Gabai examples are in fact isotopic in $B^5$, but there might in principle be more complicated examples that remain nonisotopic even when pushed into $B^5$. \footnote{Added in proof: Daniel Hartman~\cite{HartmanB3sInB5} has in fact shown that this does not happen: any two $B^3$'s in $S^4$ with the same boundary become isotopic after pushing their interiors into $B^5$.} However, here we play it safe by working with ``dotted disks'' instead of ``dotted circles'' and ``dotted balls'' instead of ``dotted spheres''.

The following technical lemma will be needed in our proof of Theorem~\ref{T:23to12} as preparation for a ``dot switch'' argument. In the statement of the lemma there is no mention of handles, but to set it in context, think of $\gamma^\bullet_t$ as a family of dotted circles describing $5$--dimensional $2$--handles and think of $\beta_t$ as a family of attaching $2$--spheres for $5$--dimensional $3$--handles. The lemma constructs a family $\beta^\bullet_t$ of dotted $B^3$'s which can either be seen as simply auxiliary data to the $3$--handle attaching spheres or as dotted $B^3$'s for $1$--handles. Once the components of $\beta^\bullet_t$ are interpreted as dotted $B^3$'s describing $1$--handles, then one is ready to interpret the boundary $S^1$'s $\partial \gamma^\bullet_t$ as attaching circles for $2$--handles rather than as boundaries of dotted disks for $2$--handles. The work of this lemma is to find the dotted $B^3$'s and extend the families in time, preserving cancellation where necessary, so that both the dotted $B^3$'s and dotted $B^2$'s actually form loops of embeddings rather than just paths of embeddings.

\begin{lemma} \label{L:PreparingForDotSwitch}
 Suppose we are given a pair of $1$--parameter families of embeddings $\gamma^\bullet_t : \amalg^n B^2 \into S^4$ and $\beta_t: \amalg^n S^2 \into S^4$, for $t \in [0,1]$, satisfying the following properties:
 \begin{enumerate}
  \item $\beta_0$ is an unlink of unknots, i.e. $\beta_0$ extends to an embedding of $\amalg^n B^3 \into S^4$.
  \item $\gamma^\bullet_1 = \gamma^\bullet_0$ and $\beta_1 = \beta_0$.
  \item For all $t \in [0,1]$, $\beta_t$ and $\partial \gamma^\bullet_t$ are disjoint.
  \item The $i$'th component of $\beta_0=\beta_1$ transversely intersects the $j$'th component of $\gamma^\bullet_0=\gamma^\bullet_1$ at exactly $\delta_{ij}$ points.
 \end{enumerate}
 Then there exists an extension of $\gamma^\bullet_t$ and $\beta_t$ to $t \in [0,3]$ and a $1$--parameter family of embeddings $\beta^\bullet_t : \amalg^n B^3 \into S^4$  defined for all $t \in [0,3]$, satisfying the following properties:
 \begin{enumerate}
  \item For all $t \in [0,1]$, $\gamma^\bullet_t$ is the same as the given $\gamma^\bullet_t$ and $\beta_t$ is the same as the given $\beta_t$.
  \item For all $t \in [0,3]$, $\partial \beta^\bullet_t = \beta_t$.
  \item $\gamma^\bullet_3 = \gamma^\bullet_0$ and $\beta^\bullet_3 = \beta^\bullet_0$.
  \item For all $t \in [1,3]$, the $i$'th component of $\beta_t$ transversely intersects the $j$'th component of $\gamma^\bullet_t$ at exactly $\delta_{ij}$ points.
  \item The $i$'th component of $\partial \gamma^\bullet_3 = \partial \gamma^\bullet_0$ transversely intersects the $j$'th component of $\beta^\bullet_3 = \beta^\bullet_0$ at exactly $\delta_{ij}$ points.
  \item The path $\beta^\bullet_t$ is homotopic rel $t \in \{0,3\}$ in $\Emb (\amalg^n B^3,S^4)$ to the constant path.
 \end{enumerate}

\end{lemma}

We have extended to the time range $t \in [0,3]$ because the proof naturally involves two extensions, one on $[1,2]$ and one on $[2,3]$. Of course this time parameter can and will be reparameterized as needed.

\begin{proof}
 First we construct $\beta^\bullet_t$ on $t \in [0,1]$. Because $\beta_t$ and $\partial \gamma^\bullet_t$ are disjoint for all $t \in [0,1]$, we can use the isotopy extension theorem to find an ambient isotopy $\phi_t: S^4 \to S^4$, for $t \in [0,1]$, such that $\phi_0 = \id$, $\phi_t \circ \beta_0 = \beta_t$ and $\phi_t \circ \partial \gamma^\bullet_0 = \partial \gamma^\bullet_t$. Choose any extension of $\beta_0$ to an embedding $\beta^\bullet_0: \amalg^n B^3 \into S^4$ with the property that the $i$'th component of $\partial \gamma^\bullet_0$ transversely intersects the $j$'th component of $\beta^\bullet_0$ at exactly $\delta_{ij}$ points. This can be done because the $i$'th component of $\partial \gamma^\bullet_0$ is a meridian to the $i$'th component of $\beta_0$. Then let $\beta^\bullet_t = \phi_t \circ \beta^\bullet_0$.
 
 The most immediate problem at this point is that there is no reason to expect that $\beta^\bullet_1 = \beta^\bullet_0$. We now extend both $\beta^\bullet_t$ and $\gamma^\bullet_t$ to $t \in [1,2]$ using the isotopy $\phi_t$ in reverse. More precisely, for $t \in [1,2]$, let
 \[ \beta^\bullet_t = \beta^\bullet_{2-t} = \phi_{2-t} \circ \beta^\bullet_0 = (\phi_{2-t} \circ \phi_1^{-1}) \circ \beta^\bullet_1 \]
 and let
 \[ \gamma^\bullet_t = (\phi_{2-t} \circ \phi_1^{-1}) \circ \gamma^\bullet_1. \]
 Also let $\beta_t = \partial \beta^\bullet_t$. Note that when $t = 1$, $\phi_{2-t} \circ \phi_1^{-1} = \id$, so we do have a continuous extension of both $\beta^\bullet_t$ and $\gamma^\bullet_t$ from $t \in [0,1]$ to $t \in [0,2]$. Also note that now, on $t \in [1,2]$, since both $\gamma^\bullet_t$ and $\beta^\bullet_t$ are moved by the same isotopy $\phi_{2-t} \circ \phi_1^{-1}$, we have that the $i$'th component of $\partial \beta^\bullet_t$ transversely intersects the $j$'th component of $\gamma^\bullet_t$ at $\delta_{ij}$ points for all $t \in [1,2]$.
 
 Furthermore, since $\beta^\bullet_t = \beta^\bullet_{2-t}$ for $t \in [1,2]$, we have guaranteed that $\beta^\bullet_2 = \beta^\bullet_0$ and that the path $\beta^\bullet_t$ in $\Emb(\amalg^n B^3, S^4)$ is homotopic rel $t \in \{0,2\}$ to the constant path.
 
 Now, however, we have no reason to expect that $\gamma^\bullet_2 = \gamma^\bullet_0$, and this is the last thing that we fix, using the time interval $t \in [2,3]$. Both $\gamma^\bullet_2$ and $\gamma^\bullet_0$ have the property that their $i$'th components transversely intersect the $j$'th component of $\beta_2 = \beta_0$ in $\delta_{ij}$ points. Parameterize a regular neighborhood of $\beta_0(\amalg^n S^2)$ as $\amalg^n (S^2 \times B^2)$ and note that both $\gamma^\bullet_2$ and $\gamma^\bullet_1$ are now isotopic to small meridional disks $\amalg^n (\{p\} \times B^2)$ centered at the points of intersection between $\gamma^\bullet_2$, resp $\gamma^\bullet_0$, with $\beta_0$. These isotopies can be chosen so as to preserve the property that $\gamma^\bullet_t$ and $\beta_0$ transversely intersect at $\delta_{ij}$ points, simply by shrinking the disks without moving them. Running one of these isotopies forward for $t \in [2,2.4]$ and the other one backwards for $t \in [2.6,3]$, and connecting them in $t \in [2.4,2.6]$ by moving the small meridional disks inside the neighborhood $\amalg^n (S^2 \times B^2)$, we get a path of embeddings $\gamma^\bullet_t$ for $t \in [2,3]$ from the given $\gamma^\bullet_2$ to $\gamma^\bullet_3 = \gamma^\bullet_0$. For all $t \in [2,3]$, let $\beta^\bullet_t = \beta^\bullet_2 = \beta^\bullet_3 = \beta^\bullet_0$ and let $\beta_t = \partial \beta^\bullet_t$. Since we simply shrank the disks $\gamma^\bullet_2$ into the tubular neighborhood of $\beta_2 = \beta_t = \beta_3$, then moved these disks around the neighborhood, and then expanded back out along the disks $\gamma^\bullet_3 = \gamma^\bullet_3$, we see that we did not introduce any extra intersections between $\partial \gamma^\bullet_t$ and $\beta_t$ for $t \in [2,3]$, and thus maintained the $\delta_{ij}$ intersection property.
 
 Lastly, since $\beta^\bullet_t$ is $t$--invariant for $t \in [2,3]$, then we still have the property that $\beta_t$ is homotopic rel $t \in \{0,3\}$ to the constant path in $\Emb(\amalg^n B^3,S^4)$.
\end{proof}

\begin{proof}[Proof of Theorem~\ref{T:23to12}]
 Consider a cobordism $Z$ from $S^1 \times S^4$ to $Y = S^1 \times_\phi S^4$ for some $\phi \in \Diff^+(S^4)$, built as before as a family $Z_t$ of cobordisms, such that each $Z_t$ is built by attaching n $2$--handles to $[0,1] \times S^4$ and then $n$ $3$--handles to the result, with the attaching data for the handles varying smoothly with $t$. The $2$--handles are attached along a family $\alpha_t$ of framed embeddings of $\amalg^n S^1$ into $S^4$. In fact at this point $\alpha_t = \alpha_0$ does not vary with $t$ and is a standard embedding, so that $S^4(\alpha_t)$ is canonically identified with $\#^n (S^2 \times S^2)$. The $3$--handles are attached along a family $\beta_t$ of framed embeddings of $\amalg^n S^2$ into $\#^n (S^2 \times S^2) = S^4(\alpha_t)$, with $\beta_0 = \beta_1 = \amalg^n (S^2 \times \{p\})$, and with each $\beta_t$ disjoint from $\amalg^n (S^2 \times \{p'\})$.
 
 We would like to show that there is also a cobordism $Z'$ from $S^1 \times S^4$ to $Y = S^1 \times_\phi S^4$ for the same $\phi \in \Diff^+(S^4)$, but now built as a family $Z'_t$ of cobordisms such that each $Z'_t$ is buit by attaching first $n$ $1$--handles to $[0,1] \times S^4$ and then $n$ $2$--handles to the result, and such that the attaching maps for the $1$--handles are $t$--invariant. This will prove the theorem.

 Now, given our handle attaching data $\alpha_t$ and $\beta_t$ used to build $Z$, since the $\alpha_t$'s are invariant in $t$, and $\alpha_t = \alpha_0$ bounds a fixed collection of framed disks $\delta_0$, we can instead represent the $2$--handles by a (for now, $t$--invariant) $t$--parametrized family of $n$ dotted disks $\gamma^\bullet_t: \amalg^n B^2 \into S^4$. Note that $\partial \gamma^\bullet_t \neq \alpha_t$, but instead $\partial \gamma^\bullet_t$ is a linking circle to $\delta_0$. Also, we insist on maintaining the subscript $t$ even though these are $t$--invariant because we will shortly modify the family so as to lose $t$--invariance. Now the instructions for building the cobordism $Z$ are to build each $Z_t$ from $[0,1] \times S^4$ by pushing the interiors of the disks $\gamma^\bullet_t$ from $\{1\} \times S^4$ into the interior of $(1-\epsilon,1] \times S^4$ and removing their neighborhoods, and then attaching $3$--handles along $\beta_t$. Note that, after carving out the disks but before attaching the $3$--handles, the upper boundary of this cobordism $Y_t$ is the surgered $4$--manifold $S^4(\partial \gamma^\bullet_t)$. In other words, when looking at the $4$--dimensional boundary, we cannot tell whether we carved out the dotted disks or attached $2$--handles along their boundaries, because the resulting surgeries are the same.
 
 As noted in the preamble to this proof discussing $5$--dimensional dotted circle and sphere notation, because each $\beta_t$ is disjoint from $\amalg^n (S^2 \times \{p'\})$, we can isotope the family $\beta_t$ so that it never goes over the surgered region of $S^4(\partial \gamma^\bullet_t)$, and thus the entire handle attaching data now lives in $S^4$. Thus we can now describe each $Y_t$, and thus $Z$, via data entirely lying in $S^4$, i.e. $\gamma^\bullet_t : \amalg^n B^2 \into S^4$ and $\beta_t : \amalg^n S^2 \into S^4$. The only intersections occur between $\amalg^n S^2$ and the interiors of the disks $\amalg^n B^2$. At times $t=0$ and $t=1$, each $S^2$ intersects its corresponding $B^2$ transversely once and is disjoint from all the other $B^2$'s, i.e. the spheres and disks are in ``cancelling position''
 
 Our goal is now to ``switch the dots from the circles to the spheres'', i.e. to think of $\beta_t$ as being dotted spheres, thus corresponding to $5$--dimensional $1$--handles, and to think of $\partial \gamma^\bullet_t$ as attaching circles for $2$--handles, rather than dotted circles describing $2$--handles. Furthermore, we need to arrange that we actually end up with a loop of dotted $B^3$'s after this extension. This was the point of Lemma~\ref{L:PreparingForDotSwitch}. To use this result, homotope our given loops of embeddings so as to arrange that $\gamma^\bullet_t$ and $\beta_t$ are $t$--invariant on the time interval $t \in [1-\epsilon,1]$ for some small $\epsilon > 0$. Apply Lemma~\ref{L:PreparingForDotSwitch}, but reparameterizing the time parameters so that the time interval $[0,1-\epsilon]$ here is mapped to the time interval $[0,1]$ in the lemma, so that the time interval $[1,3]$ in the lemma is mapped to the time interval $[1-\epsilon,1]$ in this proof.
 
 This gives us $\gamma^\bullet_t$ and $\beta^\bullet_t$ defined for $t \in [0,1]$ such that:
 \begin{enumerate}
  \item $\gamma^\bullet_1 = \gamma^\bullet_0$ and $\beta^\bullet_1 = \beta^\bullet_0$.
  \item $\gamma^\bullet_t$ agrees with the original $\gamma^\bullet_t$ on $[0,1-\epsilon]$.
  \item $\partial \beta^\bullet_t$ agrees with the original $\beta_t$ on $[0,1-\epsilon]$.
  \item If we interpret $\gamma^\bullet_t$ as dotted disks for $2$--handles and interpret $\partial \beta^\bullet_t$ as attaching $S^2$'s for $3$--handles, then these give cancelling $(2,3)$--pairs for all $t \in [1-\epsilon,1]$.
  \item If instead we interpret $\beta^\bullet_t$ as dotted $B^3$'s for $1$--handles and interpret $\partial \gamma^\bullet_t$ as attaching circles for $2$--handles, then these give cancelling $(1,2)$--pairs for $t \in \{0,1\}$. 
 \end{enumerate}
 Thus we can use the fact that the dotted disks $\gamma^\bullet_t$ and the attaching $S^2$'s $\partial \beta^\bullet_t$ are in cancelling position for $t \in [1-\epsilon,1]$ to modify our fiberwise handle decomposition of $Z$ so that the new handle decomposition is in fact given by the new families $\gamma^\bullet_t$ and $\partial \beta^\bullet_t$ for all $t \in [0,1]$. One way to see this is to first cancel the original $(2,3)$--handle pairs on the time interval $[1-\epsilon,1]$, so that we have no handles on that time interval, with deaths of cancelling pairs at $t=1-\epsilon$ and births at $t=1=0$. Then we can introduce a birth slightly after $t=\epsilon$, let the $(2,3)$--pairs move between this time and until slightly before $t=1$ following the new $\gamma^\bullet_t$ and $\partial \beta^\bullet_t$, and then have the pairs die at the time slightly before $t=1$. Because the pairs are in cancelling position over this entire time interval, this family of handle decompositions does not change the bundle $Z$. Finally we cam merge the births and deaths at $t=\epsilon$ and at $t=1=0$ to get the desired result.

 Now ``switch all the dots from the $\gamma$'s to the $\beta$'s''. In other words, build a new fiberwise cobordism $Z'$ from $S^1 \times S^4$ to some $Y'$ with dotted $1$--handles given by $\beta^\bullet_t$ and $2$--handles attached along the circles $\partial \gamma^\bullet_t$.  Because the surgered $4$--manifolds produced by the dotted ball, respectively disk, constructions are the same as those produced by attaching handles along the boundaries of the balls, respectively disks, in fact the top boundary $Y'$ of our new fiberwise cobordism is the same as our original top boundary $Y = S^1 \times_\phi S^4$.

 Finally, because we arrange that the loop $\beta^\bullet_t$ is homotopically trivial in $\pi_1(\Emb(\amalg^n B^3,S^4))$, an ambient isotopy argument as in the proof of Proposition~\ref{P:AllTheCerfTheory} can now be used to arrange that $\beta^\bullet_t$ is independent of $t$, and thus the entire construction is governed by the loop of framed circles $\partial \gamma^\bullet_t$.

\end{proof}

\section{From many $(1,2)$ pairs to a single $(1,2)$ pair}

We now know that diffeomorphisms of $S^4$ that can be realized by a family of cobordisms built with $n$ $(2,3)$--handle pairs can also be realized by a family of cobordisms built with $n$ $(1,2)$--handle pairs, as long as the original governing loop of embeddings of $\amalg^n S^2$ into $\#^n (S^2 \times S^2)$ lies in $\widehat{\mathcal{S}_n}$. Before we get to Montesinos twins, we need now to show that every diffeomorphism of $S^4$ that can be a realized by a family of cobordisms built with $n$ $(1,2)$--handle pairs can be realized by a family built with a single $(1,2)$--handle pair.

\begin{theorem} \label{T:Many12toOne12}
 For any $n$, $\mathcal{FH}_{1,n}(\pi_1(\mathcal{FS}_{1,n}(S^4))) = \mathcal{FH}_{1,1}(\pi_1(\mathcal{FS}_{1,1}(S^4)))$.
 
\end{theorem}

\begin{proof}
We begin again with a cobordism $Z$ from $S^1 \times S^4$ to $S^1 \times_\phi S^4$ built as a family of cobordisms $Y_t$, each $Y_t$ built by attaching $n$ fixed standard $1$--handles to $[0,1] \times S^4$ followed by $n$ ``moving'' $2$--handles governed by a loop of embeddings $\alpha_t: \amalg^n (S^1 \times B^3) \into \#^n (S^1 \times S^3)$. Cancelling the $(1,2)$ pairs at time $t=0 \sim 1$, we revert to the Cerf theoretic perspective to get a family $(f_t,V_t)$ of Morse functions with gradient-like vector fields on $[0,1] \times S^4$ interpolating from $f_0$, which is projection onto $[0,1]$, to $f_1$, which is the pullback of $f_0$ via some pseudoisotopy $\Phi: [0,1] \times S^4 \to [0,1] \times S^4$ from $\id_S^4$ to $\phi$. The graphic now looks like Figure~\ref{F:Nice12Cerf}, exactly as in Figure~\ref{F:Nice23Cerf} except that now the critical points are of index $1$ and $2$; there are still no handle slides.
 \begin{figure}
  \labellist
  \small\hair 2pt
  \pinlabel $1$ [tr] at 45 20 
  \pinlabel $1$ [tr] at 49 24 
  \pinlabel $1$ [tr] at 56 31 
  \pinlabel $2$ [br] at 45 67 
  \pinlabel $2$ [br] at 49 63 
  \pinlabel $2$ [br] at 56 55 
  \endlabellist
  \centering
  \includegraphics[width=10cm]{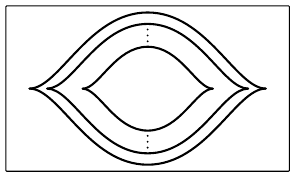}
  \caption{A nice Cerf graphic with only critical points of index $1$ and $2$.}
  \label{F:Nice12Cerf}
\end{figure}

Theorem~2.1.1 of Chenciner's thesis~\cite{Chenciner}, restated as Hatcher and Wagoner's Proposition~1.4 on p.177 of~\cite{HatcherWagoner}, asserts that, given a $1$--parameter family $f_t$ of Morse functions on $[0,1] \times X$ where $X$ is an $m$--manifold, if the Cerf graphic contains a swallowtail involving critical points of index $i$ and $i+1$ as in the left of Figure~\ref{F:Swallowtail}, with $i \leq m-3$, then the swallowtail can be cancelled to give the graphic on the right in Figure~\ref{F:Swallowtail}. This applies in our setting because $m=4$ and $i=1 = 4-3$. We use this to reduce the number of $(1,2)$ pairs using the main idea of Proposition~4 on p.217 of~\cite{HatcherWagoner}, as in the figure on the top of p.218 of~\cite{HatcherWagoner}. We essentially reproduce this figure here in Figure~\ref{F:2To1BySwallowtail} which shows how to reduce a nested pair of birth-deaths of $1$--$2$ handles to a single pair. (The other elementary moves are introducing a swallowtail, which can always be done, and merging a death with a birth, which can always be done if level sets are connected, which they are in our case.)
 \begin{figure}
  \labellist
  \small\hair 2pt
  \pinlabel $i$ [t] at 46 30
  \pinlabel $i$ [t] at 46 30
  \pinlabel $i+1$ [t] at 12 60
  \pinlabel $i+1$ [t] at 78 60
  \pinlabel $i+1$ [t] at 160 60
  \endlabellist
  \centering
  \includegraphics[width=10cm]{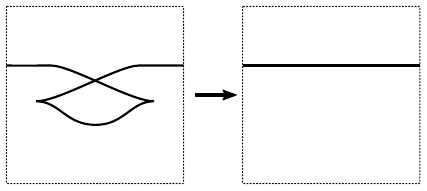}
  \caption{Eliminating a swallowtail.}
  \label{F:Swallowtail}
\end{figure}
 \begin{figure}
  \includegraphics[width=12cm]{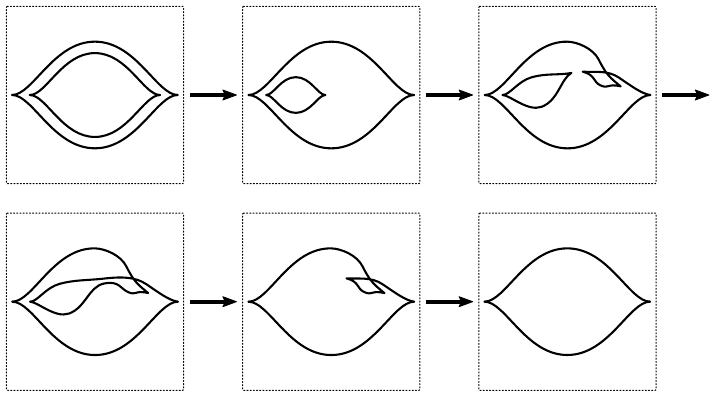}
  \caption{Using a swallowtail to turn a nested pair of ``eyes'' into a single eye.}
  \label{F:2To1BySwallowtail}
\end{figure}

Repeating this we can turn $n$ nested $1$--$2$ ``eyes'' into a single nested $1$--$2$ ``eye'', and then we can merge the birth again at $t=0 \sim 1$. Note that in fact we could have left this last (bottom-most) index $1$ critical point completely unchanged in this whole process, we even did not need to cancel it with its cancelling $2$--handle at the beginning. Thus we can easily arrange that this last $1$--handle is still stationary, i.e. its attaching map does not move with $t$. This shows that this cobordism can be built with a single fixed standard $1$--handle followed by a single moving $2$--handle whose attaching map is given by a loop of embeddings $S^1 \times B^3 \into S^1 \times S^3$. Therefore $[\phi] \in \mathcal{FH}_{1,1}(\pi_1(\mathcal{FS}_{1,1}(S^4)))$.
 
\end{proof}

\begin{remark}
 In case it is needed in another context, the general version of this theorem is that, if $k \leq m-3$ and $X$ is an $m$--manifold, then for any $n$, $\mathcal{FH}_{k,n}(\pi_1(\mathcal{FS}_{k,n}(X))) = \mathcal{FH}_1(\pi_1(\mathcal{FS}_{k,1}(X)))$.
\end{remark}

\section{Twists along Montesinos twins} \label{S:Montesinos}

As a consequence of the preceding two theorems we now know that any diffeomorphism of $S^4$ arising as the monodromy of the top of a cobordism constructed as above from a loop of $n$ $2$--spheres in $\#^n (S^2 \times S^2)$ which remain disjoint from a parallel copy $\amalg^n (S^2 \times \{p'\})$ of the basepoint embedding $\amalg^n (S^2 \times \{p\})$ is isotopic to a diffeomorphism arising from a loop of embeddings of a single circle in $S^1 \times S^3$. This is summarized as:
\begin{corollary}
 \[\mathcal{H}_{2,\infty}(\imath_*(\pi_1(\widehat{\mathcal{S}}_{2,\infty}(S^4)))) = \mathcal{FH}_{1,1}(\pi_1(\mathcal{FS}_{1,1}(S^4))) \]
\end{corollary}

Our next goal, which will complete the proof of Theorem~\ref{T:TwinsGenerate}, is to show that 
\[\mathcal{FH}_{1,1}(\pi_1(\mathcal{FS}_{1,1}(S^4)))\] 
is generated by twists along Montesinos twins as advertised. Recall that a Montesinos twin is a pair $W=(R,S)$ of embedded, oriented $2$--spheres $R,S \subset S^4$ intersecting transversely at two points. We asserted in the introduction that the boundary of a regular neighborhood of $R \cup S$ is a $3$--torus $S^1_l \times S^1_R \times S^1_S$. We will now give an explanation of this fact tailored to the proof to follow.

Consider the $4$--manifold $X = S^1 \times B^2 \times S^1$ with boundary $M = S^1 \times S^1 \times S^1$. Let $\Gamma$ be the core $2$--torus $S^1 \times \{0\} \times S^1 \subset X$ and let $C$ be the circle $\{1\} \times \{0\} \times S^1 \subset \Gamma$, with the obvious framing coming from all the product structures involved. Now let $X_C$ be the result of surgery along $C$. Since $C \subset \Gamma$, this also surgers $\Gamma$ along $C$ and the resulting surface $\Gamma_C$ is a $2$--sphere, which we will call $R$. Since surgering along a circle $C$ replaces an $S^1 \times B^3$ with $S^2 \times B^2$, we also get a new $S^2$ coming from the surgery, which we will call $S$, and we note that $S$ which intersects $R$ at exactly two points, with opposite sign. Thus in fact $X_C$ is a regular neighborhood of two $2$--spheres $R \cup S$, each with trivial normal bundle, intersecting each other twice with opposite signs. In other words, $X_C$ is a regular neighborhood of a Montesinos twin $W = R \cup S$.

Furthermore, looking at the boundary $3$--manifold $M = S^1 \times S^1 \times S^1 = \partial S^1 \times B^2 \times S^1 = \partial X_C$, we see that the second factor $\{1\} \times S^1 \times \{1\}$ bounds a disk transversely intersecting $R$ once, and is therefore characterized by being a meridian to $R$. Likewise the third factor $\{1\} \times \{1\} \times S^1$ bounds a disk (after the surgery) which intersects $S$ transversely once, and is therefore characterized by being a meridian to $S$. The first factor is not uniquely characterized by this construction, but if the Montesinos twin is embedded in $S^4$ then it can be uniquely characterized by being a longitude, i.e. being homologically trivial in the complement of the twin.

\begin{proof}[Proof of Theorem~\ref{T:TwinsGenerate}]
Recall that $\mathcal{FS}_{1,1}(S^4))$ is the space of framed embeddings of $S^1$ in $S^1 \times S^3$ while $\mathcal{S}_{1,1}(S^4)) = \Emb(S^1,S^1\times S^3)$ is the space of unframed embeddings of $S^1$ in $S^1 \times S^3$. As noted at the end of Section~\ref{S:LoopsToDiffs}, we might worry that the homomorphism $\pi_1(\mathcal{FS}_{1,1}(S^4))) \to \pi_1(\mathcal{S}_{1,1}(S^4)))$ is not surjective, since there are two possible framings of a circle in a $4$--manifold. However, Budney and Gabai~\cite{BudneyGabai} give explicit representatives of generators for $\pi_1(\Emb(S^1,S^1\times S^3))$ all of which can be seen to lift to framed loops of embeddings, and thus the map is surjective so we do not need to worry about framings anymore. There is presumably a more direct way to see this, the point being that there is no loop of embeddings of $S^1$ in $S^1 \times S^3$ which switches the two framings of $S^1$.

For the remainder of this proof, we will use the less obscure notation $\Emb(S^1,S^1\times S^3)$ to refer to the space $\mathcal{S}_{1,1}(S^4)$, the latter more complicated notation only being helpful when placing things in the much more general context of Section~\ref{S:LoopsToDiffs}. Also note that in Section~\ref{S:LoopsToDiffs} we punctured the target space of our embeddings, but here we drop the puncture for simplicity. The point is that, at the level of $\pi_1$, the puncture is irrelevant since loops of circles are $2$--dimensional while homotopies of loops of circles are $3$--dimensional, so since the ambient space is $4$--dimensional everything can be assumed to miss a point.

In fact~\cite{BudneyGabai} shows that every class $a \in\pi_1(\Emb(S^1,S^1\times S^3))$ can be represented by a loop $\alpha_t : S^1 \into S^1 \times S^3$ of embeddings such that the associated map $\Gamma: S^1 \times S^1 \to S^1 \times S^3$ given by $\Gamma(t,s) = \alpha_t(s)$ is itself an embedding. Thus we have an embedded torus $\Gamma: S^1 \times S^1 \into S^1 \times S^3$ such that $\Gamma(\{0\} \times S^1)$ is the basepoint embedding $C = S^1 \times \{p\}$. Surgery along $C$ applied to the pair $(S^1 \times S^3, \Gamma)$ yields $(S^4,R)$ for some embedded $2$--sphere $R \subset S^4$, and the $2$--sphere $S$ dual to the surgery circle is an unknotted sphere $S \subset S^4$ such that $(R,S)$ forms a Montesinos twin. Furthermore, the boundary $\partial \nu(R \cup S)$ of a neighborhood of this twin in $S^4$ is the same as the boundary of a tubular neighborhood of $\Gamma(S^1 \times S^1)$ in $S^1 \times S^3$. When this $3$--torus is parametrized as $S^1_l \times S^1_R \times S^1_S$ as in the introduction, we see that the $S^1_l$ parameter corresponds to the $t$ parameter in $\Gamma(t,s) = \gamma_t(s)$, that the $S^1_R$ direction corresponds to the $s$--parameter, and that the $S^1_S$ direction corresponds to the boundary of the disk factor in the tubular neighborhood $\nu(\Gamma(S^1 \times S^1)) \cong D^2 \times S^1 \times S^1$. 
 
Because $\Gamma$ is embedded, it is relatively easy to see what $\mathcal{H}_1([\gamma_t])$ looks like. We need an ambient isotopy $\phi_t$ of $S^1 \times S^3$ with $\phi_0 = \id$, $\phi_t \circ \gamma_0 = \gamma_t$ and $\phi_1$ equal to the identity on a neighborhood of $C$. This is the ``circle pushing'' map we get by dragging the circle around the embedded torus and back to its starting position. This can happen entirely in a tubular neighborhood $D^2 \times S^1 \times S^1$ of $\Gamma$, by spinning in the $t$ direction more and more as we move towards the center of $D^2$, which we state explicitly as follows: Let $(r,\theta)$ be polar coordinates on $D^2$, and let $(t,s)$ be coordinates as before on $S^1 \times S^1$. Choose a smooth non-increasing function $T: [0,1] \to [0,2\pi]$ which is $1$ on $[0,1/4]$, $0$ on $[3/4,1]$, and let $\phi_t(r,\theta,t,s) = (r,\theta,t+T(r),s)$. From this it is clear that $\phi_1$ is the identity on $r \in [0,1/4]$ and $r \in [3/4,1]$, and on the intermediate $[1/4,3/4] \times S^1 \times S^1 \times S^1$ is equal to a Dehn twist on $[1/4,3/4] \times S^1$ crossed with the identity in the remaining $S^1 \times S^1$ direction. Back in $S^4$ this is exactly the twist $\tau_W$ along the twin $W = (R,S)$.

In fact~\cite{BudneyGabai} establishes an isomorphism
\[ W_1 \times W_2 : \pi_1(\Emb(S^1,S^1 \times S^3),C) \to \Z \oplus \Lambda_3^1 \]
where $\Lambda_3^1$ is a free abelian group on a countably infinite generating set. The $\Z$ factor in $\Z \oplus \Lambda_3^1$ is given by the loops of $S^1$--reparametrizing embeddings $\gamma_t(s) = \gamma_0(s+nt)$, and it is easy to see that $\mathcal{H}_1$ applied to such a loop of embeddings is isotopic to $\id_{S^4}$, i.e. this $\Z$ factor is in the kernel of $\mathcal{H}_1$. Modulo this $\Z$ factor, Figure~4 in~\cite{BudneyGabai} gives the first two tori $T(1)$ and $T(2)$ in an obvious family $T(i)$ of tori in $S^1 \times S^3$ which give the countably infinite generating set corresponding to $\Lambda_3^1$. We draw $T(3)$ in Figure~\ref{F:FromBudneyGabaiToMontesinos}. In this figure, the circle $C \subset S^1 \times S^3$ is represented as a red vertical line on the far right side of the torus. The torus $T(n)$ is just like this but wraps $n$ times around the $S^1$ direction. Surgering along $C$ yields our Montesinos twins $W(i) = (R(i),S)$.
\begin{figure}
  \labellist
  \small\hair 2pt
  \pinlabel $C$ [r] at 230 150
  \endlabellist
  \centering
  \includegraphics[width=6cm]{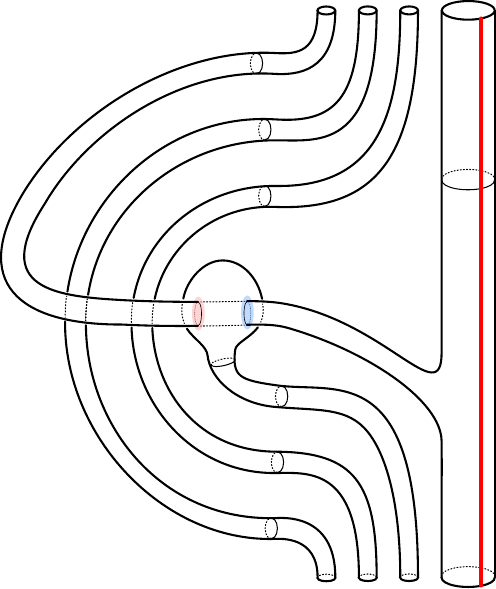}
  \caption{The embedded torus $T(3)$ in $S^1 \times S^3$, the obvious next member of the family of tori described in Figure~4 of~\cite{BudneyGabai}. The top is glued to the bottom, and horizontal slices are $S^3$'s, with the ``time'' coordinate indicated in red/blue shading, as in Figure~\ref{F:snake}. Here we have exaggerated certain features of this torus and deformed somewhat from the drawings in Figure~4 of~\cite{BudneyGabai} so that the connection with the Montesinos twin $W(3)=(R(3),S)$ in Figure~\ref{F:snake} is visually apparent. Surgering along the red circle $C$ collapses the vertical cylinder on the right into a sphere (the tail of the snake), with the dual sphere to $C$ becoming the tail-piercing sphere $S$.}
  \label{F:FromBudneyGabaiToMontesinos}
\end{figure}

\end{proof}

\bibliographystyle{plain}
%
%
\bibliography{S4Diffeos}

\end{document}